\documentclass[11pt,reqno]{amsart}

\usepackage[utf8]{inputenc}
\usepackage[T1]{fontenc}
\usepackage[italian,french,english]{babel}

\usepackage{amsthm}
\usepackage{amsmath}
\usepackage{amssymb}
\usepackage{bbm}
\usepackage{enumitem}
\usepackage{graphicx}
\usepackage{listings}
\usepackage{mathrsfs}
\usepackage{mathtools}
\usepackage{multirow}
\usepackage{pict2e}
\usepackage{stmaryrd}
\usepackage{tikz}
\usepackage[all,cmtip]{xy}

\usepackage{xcolor}
\definecolor{marin}{rgb}   {0.,   0.1,   0.5} 
\definecolor{rouge}{rgb}   {0.8,   0.,   0.} 
\definecolor{sepia}{rgb}   {0.4,   0.25,   0.} 
\definecolor{mag}{rgb}   {0.3,   0,   0.3} 
\usepackage[colorlinks,citecolor=sepia,linkcolor=marin,urlcolor=black,bookmarksopen,bookmarksnumbered]{hyperref}
\newcommand{\hyperL}[1]{\hyperref{#1}{}{}{#1}}

\usepackage[pagewise,mathlines]{lineno}

\newtheorem{theorem}{Theorem}[section]
\newtheorem{corollary}[theorem]{Corollary}
\newtheorem{lemma}[theorem]{Lemma}
\newtheorem{proposition}[theorem]{Proposition}

\newtheorem{remark}[theorem]{Remark}

\numberwithin{equation}{section}

\newcommand{\mi}{\mathbf{i}}
\newcommand{\zero}{\mathbf{0}}
\newcommand{\md}{\mathrm{d}}

\newcommand{\Exp}[1]{\exp\left(#1\right)}
\newcommand{\mth}{\mathrm{th}}
\newcommand{\id}{\mathbf{id}}

\newcommand{\bA}{\mathbb{A}}

\newcommand{\bP}{\mathbb{P}}
\newcommand{\bH}{\boldsymbol{H}}
\newcommand{\bR}{\mathbb{R}}
\newcommand{\bC}{\mathbb{C}}
\newcommand{\bT}{\mathbb{T}}
\newcommand{\bN}{\mathbb{N}}
\newcommand{\bZ}{\mathbb{Z}}
\newcommand{\M}{\mathcal{M}}
\newcommand{\N}{\mathcal{N}}
\newcommand{\B}{\mathcal{B}}
\newcommand{\cH}{\mathcal{H}}
\newcommand{\U}{\mathcal{U}}
\newcommand{\mL}{\mathscr{L}}
\newcommand{\wa}{\widehat{\alpha}}
\newcommand{\meas}{\mathrm{Meas}}

\title[On the reducibility of  the 1d quantum harmonic oscillator]
{On the reducibility of the 1d quantum harmonic oscillator with a quasi-periodic bounded potential}

\author{Emanuele Haus}

\address{\small Dipartimento di Matematica e Fisica, Universit\`a degli Studi Roma Tre, Largo San Leonardo Murialdo 1, 00146 Roma, Italy}

\email{emanuele.haus@uniroma3.it\,\textrm{(E. Haus)}}

\author{Zhiqiang Wang
}

\address{\small Chern Institute of Mathematics and LPMC, Nankai University, Tianjin 300071, China}

\email{zqwang@nankai.edu.cn \& zhiqiang.wang@uniroma3.it\,\textrm{(Z. Wang)}}



\begin{document}

\begin{abstract} 
 Using the decay along the diagonal of the matrix representing the perturbation with respect to the Hermite basis, we prove a reducibility result in $L^2(\mathbb{R})$ for the one-dimensional quantum harmonic oscillator perturbed by time quasi-periodic potential, via a KAM iteration. The potential is only bounded (no decay at infinity is required) and its derivative with respect to the spatial variable $x$ is allowed to grow at most like $|x|^\delta$ when $x$ goes to infinity, where the power $\delta<1$ is arbitrary.
\end{abstract} 

\maketitle

\setcounter{tocdepth}{2} 
\tableofcontents

\section{Introduction}
In this paper we consider the following linear Schrödinger equation 
\begin{equation}\label{scheqn}
\mi \partial_t \psi=-\partial_x^2\psi+x^2\psi+\epsilon V(x,\omega t)\psi, \quad \psi=\psi(t,x),\quad (t,x)\in\mathbb R^2,
\end{equation}
where $\epsilon>0$ is a small parameter, and the frequency vector $\omega\in\mathbb R^n$ 
is a parameter belonging to $\Pi:=[0,2\pi)^n$. Throughout the paper, we assume that the potential $V:\,\mathbb R\times\mathbb T^n\ni(x,\theta)\mapsto V(x,\theta)\in\mathbb R$ is $C^1$ smooth in $(x,\theta)\in\mathbb R\times\mathbb T^n$ and analytic in $\theta\in\mathbb T^n$, where $\mathbb T^n=\mathbb R^n/(2\pi\bZ)^n$ denotes the $n$ dimensional torus.
More precisely, we assume that the function $V(x,\cdot)$ extends analytically to the strip $\mathbb T^n_\sigma=\{a+b\mi\in\mathbb C^n/(2\pi\bZ)^n: |b|<\sigma\}$, where $\sigma>0$.
We also assume that $V$ is bounded and that its derivative w.r.t.\ $x$ grows at most like $|x|^\delta$, for some $\delta<1$, as $|x|$ goes to infinity\,\footnote{Here $\delta$ can be close to $1$ arbitrarily, and without loss of generality we assume $\delta>0$.}. In other words, we assume that there exist $\delta\in(0,1)$ and $C>0$ such that 
\begin{equation}\label{potentialbound}
|V(x,\theta)|\le C,\quad |\partial_xV(x,\theta)|\le C(1+|x|)^\delta,\quad\forall~(x,\theta)\in\mathbb R\times\mathbb T^n_\sigma.
\end{equation}
Motivated by the reducibility results in \cite{bambusi01,liu10}, Eliasson\,\footnote{Eliasson mentioned this open problem at the Saint \'Etienne de Tin\'ee winter school in 2011.} in 2011 asked the question about the reducibility of the quantum harmonic oscillator \eqref{scheqn}, if the smooth perturbation $V(x,\omega t)$ is only bounded, i.e.\ if the upper bound on the derivative w.r.t.\ $x$ is erased from condition \eqref{potentialbound} above. However, this problem is still open at present.

As is well known, KAM techniques are the most useful and powerful tool to answer this type of questions. Spectral properties of the linear operator clearly play a crucial role, and the frequencies that one has to handle along the reducibility scheme, as well as the final frequencies, have the form $\Lambda_i=\lambda_i+\widetilde\lambda_i$, where $\lambda_i$ are the unperturbed frequencies and $\widetilde\lambda_i$ are the perturbations on the diagonal of the linear operator, which we call \emph{tails} of the spectrum.
To impose the measure estimates along the KAM scheme, one typically needs (see for instance Assumption B in the classical work \cite{poschel96}) some kind of decay in the tails of the spectrum. For instance, reducibility has been proved when perturbations decay like a polynomial in \cite{grebert11}. Later, Wang-Liang \cite{wang17} improved the decay to a logarithmic type. In fact, the above two results have both been extended into high-dimensional cases in \cite{grebert19,liang22Cal} respectively. Intuitively, one may expect the answer to Eliasson's question to be negative, because it is known that there will be no decay in the tails of the spectrum if the potential $V(x,\omega t)$ is merely bounded, and no proper measure estimates could be imposed along the KAM iteration. This means that, if the answer to Eliasson's question is positive, one needs totally new ideas to prove it, which looks like a very hard task.

A key remark is that the decay in the tails of the spectrum is used in order to impose the second Melnikov conditions, where the eigenvalues will appear in pairs like $\Lambda_i-\Lambda_j$. Therefore, one only needs the decay for differences of couples of tails, like $\widetilde\lambda_i-\widetilde\lambda_j$. Enlightened by this thought, in \cite{liang22Non} Liang and one of the authors proved a reducibility result assuming that the derivative (w.r.t. $x$) of perturbation decreases at least like $|x|^{-1}$ as $x$ goes to infinity. In fact, the idea dates back to the work of Faou and Grébert in \cite{faou25}. As shown in \eqref{potentialbound} above, now the derivative is allowed to increase not faster than $|x|^\delta$ at infinity, where the power $\delta<1$ can be as close to 1 from below as desired.

Before stating our main theorem, we briefly recall also some important reducibility results about higher-dimensional PDEs, while the literature for the one-dimensional case is extremely vast. In this context, Eliasson-Kuksin \cite{eliasson09} first proved a reducibility result for the following linear Schrödinger equation on a $d$-dimensional torus with a non-autonomous potential which is  quasi-periodic in time
\[
\mi u-\Delta u+\epsilon V(\varphi_0+\omega t,x;\omega )u=0,\quad u=u(t,x),\quad x\in\bT^d.
\]
However, the key idea was from their other work \cite{eliasson10}, where they obtained the existence of quasi-periodic solutions for nonlinear Schrödinger equations of the form
\[
\mi u-\Delta u+V(x)*u+\epsilon\frac{\partial F}{\partial\bar u}(x,u,\bar u)=0,\quad u=u(t,x),\quad x\in\bT^d.
\]
In particular, for the quantum harmonic oscillator equation, we would like to mention the reducibility results with bounded perturbations by Grébert-Paturel \cite{grebert19}, Liang-Wang \cite{liang19}, Liang-Wang \cite{liang22Cal}. As for the unbounded perturbations, the first reducibilty result was obtained by Bambusi-Grébert-Maspero-Robert \cite{bambusi18}, where the perturbation was only allowed to be a polynomial of degree 2 in $(x,-\mi\nabla_x)$ with time quasi-periodic coefficients. Until now, the reducibility problem for higher-dimensional quantum harmonic oscillators perturbed by a generic unbounded potential is still open.
Concerning reducibility results for other PDEs, we refer the reader to the works \cite{baldi21,bambusi19,corsi18,feola19,feola21,feola20,montalto19} and the references therein.

Finally, we would like to mention a brilliant stability result for an abstract linear Schrödinger equation of the form
\begin{equation}\label{abstractSch}
\mi\partial_tu=H_0u+V(t)u,\quad u=u(t,x),\quad x\in\bR^d.
\end{equation}
where (i) $H_0$ is the Laplace
operator of order 2 on a Zoll manifold and $V(t)$ a pseudo-differential operator of order smaller than 2; (ii) $H_0$ is the (resonant or nonresonant) harmonic oscillator of order 1 in $\bR^d$ and $V(t)$ a pseudo-differential operator of order smaller than 1, quasi-periodic in time. In \cite{bambusi21}, for any $\varepsilon>0$,  Bambusi-Grébert-Maspero-Robert proved a $(1+|t|)^\varepsilon$ upper bound on the growth of Sobolev norms of solutions to equation \eqref{abstractSch} above.

\subsection{Main theorems and related results}

Denote by $\lambda_i:= 2i-1$ the eigenvalues of the unperturbed quantum harmonic oscillator.
We state our main results as follows.

\begin{theorem}\label{mainThm}
Assume that the potential $V$ satisfies \eqref{potentialbound}. There exists $\epsilon_*>0$ such that for all $\epsilon\in(0,\epsilon_*)$  there is  a Cantor subset $\Pi_\epsilon\subset\Pi:=[0,2\pi)^n$ of asymptotically full measure such that for all $\omega\in\Pi_\epsilon$, the linear Schrödinger equation \eqref{scheqn} is reducible in $L^2(\bR)$ to a linear equation with constant coefficients (w.r.t.\ the time variable $t$).

Furthermore, letting $p\in[0,2]$, for all $\omega\in\Pi_\epsilon$ there is a linear isomorphism $\Psi_{\omega,\epsilon}\in\B(\cH^p)$, unitary on $L^2$ and real-analytically depending on  $\theta\in\bT^n_{\sigma/2}$, such that $t\mapsto\psi(t,\cdot)\in\cH^p$ satisfies the original Schrödinger equation \eqref{scheqn} if and only if $t\mapsto\phi(t,\cdot):=\Psi_{\omega,\epsilon}^{-1}(\omega t)\psi(t,\cdot)\in\cH^p$ solves the following autonomous equation 
\[
\mi\partial_t\phi=H^\infty\phi,\quad H^\infty=diag\{\lambda_i^\infty\}_{i\in\bN}.
\]
More precisely, there is a constant $C>0$ such that 
\begin{equation}\label{mainThmEstimate}
\begin{gathered}
	\meas(\Pi\backslash\Pi_\epsilon)\le C\epsilon^{\frac{1-\delta}{17(5-\delta)}},\\
|\lambda_i^\infty-\lambda_i|\le C\epsilon,\quad\forall\,i\in\bN,\\
||\Psi^\pm_{\omega,\epsilon}(\theta)-\id||_{\B(\cH^p)}\le C\epsilon^{2/3},\quad\forall\,\theta\in\bT^n_{\sigma/2}.
\end{gathered}
\end{equation}
\end{theorem}

As a consequence, we directly obtain the following corollaries (see \cite{grebert19} for the detailed proof) on the stability of the  solution and on the spectrum of  the associated Floquet operator
\begin{equation}\label{floquet op}
  K_F:=-\mi\sum_{j=1}^n\omega_j\partial_{\theta_j}-\partial_x^2+x^2+\epsilon V(\theta).  
\end{equation}

\begin{corollary}
Assume that the potential $V$ satisfies \eqref{potentialbound}. There exists $\epsilon_*>0$ such that for all $\epsilon\in(0,\epsilon_*)$ and $\omega\in\Pi_\epsilon$, the Cauchy problem of \eqref{scheqn} with $\psi(0,x)=\psi^{(0)}\in\cH^p$ has a unique solution $\psi(t,\cdot)\in C(\bR,\cH^p)$, where $p\in[0,2]$. In addition, the solution $\psi(t,\cdot)$ is almost-periodic in time and satisfies 
\[
(1-C\epsilon)||\psi^{(0)}||_{\cH^p}\le||\psi(t,\cdot)||_{\cH^p}\le(1+C\epsilon)||\psi^{(0)}||_{\cH^p},\quad\forall\,t\in\bR.
\]
\end{corollary}

\begin{corollary}
Assume that the potential $V$ satisfies \eqref{potentialbound}. There exists $\epsilon_*>0$ such that for all $\epsilon\in(0,\epsilon_*)$ and $\omega\in\Pi_\epsilon$, the Floquet operator $K_F$ defined in \eqref{floquet op} has a pure point spectrum.
\end{corollary}

\begin{remark} \label{logDecay}
Comparing the present result with the KAM iteration in \cite{liang22Non,liang22Cal}, it is likely that a slightly weaker assumption on the derivative of the potential $V$ is sufficient. It should be enough to require
\begin{equation}\label{logdecay}
	|\partial_x V(x,\theta)|\le \frac{C(1+|x|)}{\ln^{\delta'}(2+|x|)},
\end{equation}
where $\delta'=\delta'(n)>0$, because it would imply a logarithmic decay in the elements of the difference matrix of the perturbation (see Remark \ref{logdecayP} below). 
\end{remark}

\begin{remark} In the series of works \cite{bambusi18I,bambusi17II,bambusi18III}, Bambusi and Montalto  have solved the problem of reducibility of the Schrödinger operator on $\bR$ perturbed by a pseudo-differential operator. However, such perturbations exclude oscillating functions like $\cos(x)$. In particular, in \cite{bambusi17II} Bambusi mentioned that the framework of pseudo-differential operator rules out cases like $V(x,\omega t)=\cos(x-\omega t)$ in the Schrödinger equation \eqref{scheqn}. Such a potential, however, satisfies our assumption \eqref{potentialbound}. It is worth mentioning that the specific case of oscillating perturbations has already been considered in \cite{liang21}.
\end{remark}

\subsection{Paper novelties and outline}

We first recall basic facts and briefly introduce notation related to the Hermite basis. As is well known, the Hermite operator $\bH=-\frac{d^2}{dx^2}+x^2$ on the real line has a simple pure point spectrum $\lambda_i=2i-1$ for $i\ge1$ and its normalized eigenfunctions $\{h_i(x)\}_{i\ge1}$ form an orthonormal basis of $L^2$. For a function $f(x)\in L^2$, we have the Hermite expansion
\[
f(x)=\sum_{i\ge1}u_ih_i(x) \text{ with }u_i=\langle f(x),h_i(x)\rangle=\int_{\bR}f(x)h_i(x)\md x,\ \forall\,i\ge1,
\]
where $\langle\cdot,\cdot\rangle$ represents the real $L^2$ scalar product.
By abuse of language, in the whole paper we will not distinguish the function $f(x)$ and its Hermite coefficient vector $u=(u_i)_{i\ge1}$. In addition, for $s\ge0$ we identify the function space $\cH^s$ with the sequence space $\ell^2_s$ by 
\[
||f(x)||_{\cH^s}:=\sqrt{\sum_{i\ge1}i^s|u_i|^2}=:||u||_s
\]
where
\(
\cH^s:=\{f\in L^2:\langle f,\bH^s f\rangle<\infty\}
\)
and 
\(
\ell^2_s:=\{u\in\ell^2:\sum\nolimits_{i\ge1}i^s|u_i|^2<\infty\}.
\)
For $s<0$, we regard $\cH^s$ (resp. $\ell^2_s$) as the dual space of $\cH^{-s}$ (resp. $\ell^2_{-s}$). In particular, we have $L^2=\cH^0$ and $\ell^2=\ell^2_0$.

Following the classical strategy (see \cite{grebert19,grebert11,liang22Non,liang22Cal,wang17}), we consider the Schrödinger equation \eqref{scheqn} on the Hermite basis. Thus, we get an infinite-dimensional non-autonomous system of the form 
\[
\mi\dot u=(\bA+\epsilon\bP)u,
\]
where $u=(u_i)_{i\ge1}$ represents the Hermite coefficients, $\bA=diag\{1,3,5,\ldots\}$ and the perturbation matrix $\bP$ is defined by
\begin{equation}\label{definitionP}
	\mathbb P_i^j(\omega t)=\int_{\mathbb R}V(x,\omega t)h_i(x)h_j(x)\md x,\quad i,j\in\mathbb N.
\end{equation}

When dealing with bounded perturbations, one needs some kind of regularity of the perturbation matrix $\bP$ that persists in the KAM iteration. In general, we have to prove that the matrix $\bP$ is a regularizing operator, mapping from $\ell_s^2$ to $\ell_{s'}^2$ for some $s'>s$ (see \cite{kuksin87,kuksin93,poschel96}), or that its elements $\bP_i^j$ have some decay over the indexes $i,j\in\bN$ (see \cite{grebert19,grebert11,wang17}). The property that $\bP$ is a regularizing operator governs the estimates of solutions to the homological equation; on the other hand, the decay of the matrix elements directly implies a decay in the tails of the spectrum, which allows one to prove the measure estimates in the KAM iteration. Both properties are therefore essential and, as is well known, they are closely related to each other. In the first works, like \cite{kuksin87,poschel96}, the approach was based on a more abstract level on the regularizing property of the perturbation. Subsequently, more attention has been devoted to the element decay (see for instance \cite{grebert19,grebert11,liang22Non,liang22Cal,wang17}). 
 In particular, in \cite{liang22Non} Liang and one of the authors of the present paper first used the decay of the elements of the difference matrix $\Delta\bP$ to control the measure estimates. However, there was some room for improvement in exploiting the regularizing operator property. Guided by  this thought, in this paper we combine the use of operator norm and element decay to control, respectively, the solution of the homological equation and the measure estimates. This is the essential novelty of this paper. More precisely, we have the following technical novelties:
 
 (I) An improvement in the estimates of $\bP_i^j$, especially on $\Delta\bP_i^j$. In comparison with Lemma 3.2 in \cite{liang22Non}, using the uniform decay of the Hermite basis (see estimate (11) in \cite{koch05}), we prove a better regularity of the perturbation matrix $\bP$, and in particular of its associated difference matrix $\Delta\bP$ (see Lemma \ref{checkP} below).
 
 (II) A new estimate on the operator norm of solutions to the homological equation (see \eqref{homoEqn} below). Namely, if the solution $S$ belongs to $\B(\ell^2)$, then, by using the homological equation itself, we know, furthermore, that $S$ belongs to $\B(\ell_s^2)$ for all $s\in[-2,2]$, which will play a crucial role in the reduction in $\ell^2$ (resp. $L^2$ for the Schrödinger equation \eqref{scheqn}) and will also help the persistence of the difference structure of the perturbation matrix along the scheme (see Proposition \ref{homoSolution} and Lemma \ref{Algebra} below). 

 \medskip
 
Finally, we sketch the outline of the paper.
In Section 2 we will present the reducibility theorem, followed by its application to the quantum harmonic oscillator equation \eqref{scheqn} (i.e., the proof of Theorem \ref{mainThm}).
In Section 3, we prove the reducibility theorem (i.e. Theorem \ref{reducibilityThm}) via KAM tools. More precisely, we first introduce the homological equation and study its solution. Then, we perform the KAM iteration in order to finish proving Theorem \ref{reducibilityThm}. Finally, Appendix A consists of the proof of Lemma \ref{Algebra}, while Appendix B contains some auxiliary lemmas.

\subsection{Notation}
In all this paper we set $\bZ=\{0,\pm1,\pm2,\pm3,\ldots\},\,\mathbb N=\{1,2,3,\ldots\}$ and denote by $A=(A_i^j)_{i,j\in\mathbb N}$ an  infinite-dimensional matrix and by $\Delta A$ its difference matrix with entries $(\Delta A)_i^j=A_{i+1}^{j+1}-A_i^j$. For convenience,  we always write $\Delta A_i^j:=(\Delta A)_i^j$ and $i\wedge j:=\min\{i,j\}$.

For a function $f(x,\cdot)$ we always use the notation $f'(x,\cdot)$ to represent its partial derivative with respect to the spatial variable $x$. 
If not specified, we denote by $||\cdot||$ the standard linear operator norm from $L^2$ into itself (or $\ell^2$ for discrete spaces), i.e. $||F||=||F||_{\B(L^2)}:=\sup_{||g||_{L^2}=1}||Fg||_{L^2}$. Sometimes, we even omit the subscript `$L^2$', namely $||g||=||g||_{L^2}$.

For a multi-index $k\in\bZ^n$, we denote its norm by $|k|:=\sum_{i=1}^{n}|k_i|$.
Finally, for $a,b\ge0$, the notation $ a\lesssim b$ means that there is a constant $C>0$ such that $a\le Cb$. 

\section{Reducibility theorem and its application}
In this section, we will first state an abstract reducibility theorem for a non-autonomous system, quasi-periodic in time, of the form $\mi \dot u=(\bA+\epsilon\bP(\omega t))u$, where $\bA$ is the quantum harmonic oscillator represented in the Hermite basis as $\bA=diag\{1,3,5,\ldots\}$.
Then, we are going to use the abstract theorem to prove the reducibility of the one-dimensional quantum harmonic oscillator on $\bR$.
\subsection{Reducibility theorem}
 Before stating the reducibility theorem we first introduce some matrix spaces and their algebraic properties.

\subsubsection*{\bf Infinite matrices} We denote $\M:=\B(\ell^2)$, i.e.\ a matrix $A:\bN\times\bN\mapsto\bC$ belongs to $\M$ if $||A||<+\infty$, where $||\cdot||$ is the standard operator norm
from $\ell^2$ to itself.
For convenience of notation, we define the diagonal matrix $\N$ by
\[
\N_i^j=\delta_{i,j}\cdot i,\quad i,j\in\bN,
\]
where $\delta_{i,j}$ represents the Kronecker symbol. In other words, $\N=diag\{1,2,3,\ldots\}$. Now we define a subset of $\M$ as follows
\[
\M_+=\{A\in\M:||[\N,A]||<\infty\},
\]
where the Poisson bracket $[\cdot,\cdot]$ stands for the commutator of two matrices. Then we equip $\M_+$ with the norm  $||A||_+:=\max\{||A||,||[\N,A]||\}$.

As introduced in \cite{chodosh11,faou25,liang22Non}, we also need the difference structure in the matrix spaces mentioned above. To this end, we define the difference operator\,\footnote{It is a linear operator, and for the moment we do not specify its domain. In particular, one has $\Delta\N=Id$.}  $\Delta: A\mapsto \Delta A$ by
\[
(\Delta A)_i^j:=A_{i+1}^{j+1}-A_i^j,\quad i,j\in\bN.
\]
As mentioned before, we always abbreviate $(\Delta A)_i^j$ as $\Delta A_i^j$. 
Then, we introduce the matrix space $\M_\alpha$. We say that $A\in\M_\alpha$ if $|A|_\alpha:=\sup_{i,j\ge1}\{(i\wedge j)^\alpha|A_i^j|\}<\infty$, where $i\wedge j$ denotes the minimum between $i$ and $j$. Furthermore, we introduce its subspace $\M_{\alpha+}=\{A\in\M_\alpha: |[\N,A]|_\alpha<\infty\}$, equipped with the norm $|A|_{\alpha+}:=\max\{|A|_{\alpha},|[\N,A]|_{\alpha}\}$. 

Finally, we define the matrix space $\M_{\wa}$ as
\[
\M_{\wa}:=\{A\in\M: \Delta A\in\M_\alpha\},
\]
which is equipped with the norm $||A||_{\wa}:=\max\{||A||,|\Delta A|_\alpha\}$. Similarly, we define its subset $\M_{\wa+}$ as
\[
\M_{\wa+}:=\{A\in\M_+:\Delta A\in\M_{\alpha+}\},
\]
which is equipped with the norm 
\(
||A||_{\widehat{\alpha}+}:=\max\{||A||,||[\N,A]||,|\Delta A|_\alpha,|[\N,\Delta A]|_\alpha\}.
\)

\begin{remark}\label{remarkMatrix}
(1). 
The difference $\Delta$ is a bounded operator from $\M_{\wa}$ to $\M_\alpha$ by definition.

(2). 
The matrix $A\in\M_{\alpha+}$ if and only if $\sup_{i,j\ge1}\{(1+|i-j|)(i\wedge j)^\alpha|A_i^j|\}<\infty$.

(3). Clearly, we have that 
	$\M_{\wa+}\subset\M_{\wa}(\M_+)\subset\M$, and 
	$||A||_{\wa+}=\max\{||A||_+,|\Delta A|_{\alpha+}\}$.
	
(4). Although the sets $\M_{\wa}$ and $\M_{\wa+}$ look a little strange now, they will play a crucial role in the whole KAM procedure. That is because the perturbation matrix $P\in\M_{\wa}$ and the solution of the homological equation will belong to $\M_{\wa+}$. What is more important is that the KAM iteration preserves these structures.
\end{remark}

Now we state the following structure lemma and postpone its proof into appendix \ref{proofAlgebra}.

\begin{lemma}\label{Algebra}
Let $\alpha\in [0,1]$, there is a constant $C>0$ such that the following hold:

(a). If $A\in\M_{\wa+}$ and $B\in\M_{\wa}$,  then $AB,BA\in\M_{\wa}$, with $||AB||_{\widehat{\alpha}},||BA||_{\widehat{\alpha}}\le C||A||_{\wa+}\cdot||B||_{\wa}$.
	
(b). If $A,B\in\M_{\wa+}$, then $AB\in\M_{\wa+}$, with $||AB||_{\wa+}\le C ||A||_{\wa+}\cdot||B||_{\wa+}$.
	
(c). If $A\in\M_+$, then for all $s\in[-2,2]$ we have $A\in\B(\ell^2_s)$, with $||A||_{\B(\ell^2_s)}\le C||A||_+$.

(d). Letting $A=diag\{d_1,d_2,d_3,\ldots\}$\,\footnote{Note that $\Delta A$ will be still diagonal.}, if $\Delta A\in\M_\alpha$, then  $|d_i-d_j|\le|\Delta A|_{\alpha}\frac{|i-j|}{(i\wedge j)^\alpha},~\forall\,i,j\in\bN$.
\end{lemma}

\begin{remark}
From item (c) above, one observes that the difference structure is not needed to control the operator norms. However, later in the KAM iteration it will play an essential role in the measure estimates. For this reason, we keep track separately of the difference structure to make sure that the measure estimates work well, which is one of the novelties of this paper, as mentioned before.
\end{remark}

\subsubsection*{\bf Parameters} In this whole paper, the frequency $\omega$ will be considered as a parameter belonging to $\Pi:=[0,2\pi)^n$. The KAM procedure will allow us to select the `good' frequencies. All the transformations along the KAM scheme will be $C^1$ smooth with respect to $\omega$. When a map $F=F(\omega)$ is only defined on a Cantor-like subset of $\Pi$, we will have to understand this smoothness property in the sense of Whitney. Also, for convenience, we always denote $||\partial_\omega F(\omega)||:=\max_{i=1}^n\{||\partial_{\omega_i}F(\omega)||\}$.

Following the previous notation, let $\rho\in(0,\sigma]$ and $D\subset\Pi$. Denote by $\M_{\wa}(D,\rho)$ the set of $C^1$ maps $D\times\bT^n_\rho\ni(\omega,\theta)\mapsto P(\omega,\theta)\in\M_{\wa}$ real analytic in $\theta\in\bT^n_\rho$, equipped with the norm 
\[
||P||_{\wa}^{D,\rho}:=\sup_{\substack{|\Im\theta|<\rho\\ \omega\in D,l=0,1}}||\partial^l_\omega P(\omega,\theta)||_{\wa}.
\]
Naturally, the space of maps $A\in\M_{\wa}(D,\rho)$ independent of $\theta$ will be denoted by $\M_{\wa}(D)$ and equipped with the norm 
\[
||A||_{\wa}:=\sup_{\omega\in D,l=0,1}||\partial_\omega^lA(\omega)||_{\wa}.
\]
In addition, denote by $\M_{\wa+}(D,\rho)$ the subset of $\M_{\wa}(D,\rho)$ that consists of the maps $S=S(\omega,\theta)$ that take values in $\M_{\wa+}$. Similarly, we endow $\M_{\wa+}$ with the norm
\[
||S||_{\wa+}^{D,\rho}:=\sup_{\substack{|\Im\theta|<\rho\\ \omega\in D,l=0,1}}||\partial_\omega^lS(\omega,\theta)||_{\wa+}.
\]

As mentioned above, our initial system can be written in the form
\begin{equation}\label{initialSystem}
\mi\dot u=(\bA+\epsilon\bP(\omega_1t,\omega_2t,\dots,\omega_nt))u,\quad u\in\ell^2,
\end{equation}
where $\bA=diag\{1,3,5,\ldots\}$ and the perturbation matrix $\bP$ is defined by \eqref{definitionP} before. We are now ready to state our reducibility theorem, whose proof is postponed to the next section.
\begin{theorem}\label{reducibilityThm}
Assume that the perturbation matrix $\bP$ belongs to $\M_{\wa}(\Pi,\sigma)$ for some $\alpha\in(0,1]$. Then, there exists $\epsilon_*>0$ such that for all $\epsilon\in(0,\epsilon_*)$ there is a Cantor subset $\Pi_\epsilon\subset\Pi$, of asymptotically full measure, such that for all $\omega\in\Pi_\epsilon$ the linear transformation $u=\U_{\omega,\epsilon}(\theta)v$ reduces the initial equation \eqref{initialSystem} to an autonomous system with constant coefficients of the form
\[
\mi\dot v=A^\infty v,\quad A^\infty=diag\{\lambda_i^\infty\}_{i\in\bN}.
\]
Here, the coordinate transformation $\U_{\omega,\epsilon}(\theta)$ is unitary in $\ell^2$ and real analytic in $\theta\in\bT^n_{\sigma/2}$, and $\lambda_i^\infty(\omega)\in\bR$ is $C^1$ smooth in $\omega$ and $\epsilon$-close to $\lambda_i$ for all $i\in\bN$. 

More precisely, there is a constant $C>0$ such that 
\begin{equation}\label{reducibilityEstimate}
\begin{gathered}
\meas(\Pi\backslash\Pi_\epsilon)\le C\epsilon^{\frac{\alpha}{17(\alpha+2)}},\\
|\lambda_i^\infty-\lambda_i|\le C\epsilon,\quad\forall\,i\in\bN,\\
||\U_{\omega,\epsilon}^{\pm}(\theta)-\id||_{\B(\ell^2_p)}\le C\epsilon^{2/3},\quad\forall\,\theta\in\bT^n_{\sigma/2}, \quad\forall\,p\in[0,2].
\end{gathered}
\end{equation}
 \end{theorem}
 
 \begin{remark}
 (1). 
 If the unperturbed eigenvalues satisfied $|\lambda_i-\lambda_j|\ge c_1|i-j|$ and $|\lambda_{i+1}-\lambda_i|\le\frac{c_2}{i^{\beta}}$ for some positive constants $c_1,c_2,\beta$, then Theorem \ref{reducibilityThm} would still hold, replacing $\alpha$ in the first item of estimate \eqref{reducibilityEstimate} with $\min\{\alpha,\beta\}$.
 
 (2). 
 For the difference matrix $\Delta\bP$, the polynomial decay of the matrix elements is not necessary. In fact, the logarithmic decay is enough (see \cite{liang22Non,liang22Cal}), which implies Remark \ref{logDecay}. 
 \end{remark}

\subsection{Application to the quantum harmonic oscillator on $\mathbb{R}$}

As mentioned in the introduction, by the Hermite basis expansion, 
the quantum harmonic oscillator \eqref{scheqn} can be written as the infinite-dimensional system \eqref{initialSystem}. By reducibility Theorem \ref{reducibilityThm}, we only need to check that the initial perturbation $\bP$ belongs to $\M_{\wa}(\Pi,\sigma)$ for some $\alpha\in(0,1]$\,\footnote{Since we need some kind of decay for the measure estimates in the KAM iteration, hereafter we introduce the parameter $\alpha\in(0,1]$.}, which is a direct consequence of the following Lemma \ref{checkP}. Before that, we introduce two operators related to the Hermite operator $\bH=-\partial_x^2+x^2$ and state their relevant properties. From Lemma 3.1 in \cite{chodosh11} (or see \cite[Lemma 3.2]{liang22Non}) we get for any $i\ge1$
\begin{equation}\label{xDx}
\begin{aligned}
\partial_x h_i(x)&=-\sqrt{i/2}\,h_{i+1}(x)+\sqrt{(i-1)/2}\,h_{i-1}(x),\\
xh_i(x)&=\sqrt{i/2}\,h_{i+1}(x)+\sqrt{(i-1)/2}\,h_{i-1}(x).
\end{aligned}
\end{equation}
Using the same notation as in \cite{liang22Non}, let $T=\partial_x+x,\,T^\dagger=-\partial_x+x$. Then one has $T^*=T^\dagger,\,TT^\dagger=\bH+Id$ and
\[
	Th_i(x)=\sqrt{2(i-1)}h_{i-1}(x),\qquad T^\dagger h_i(x)=\sqrt{2i}h_{i+1}(x),\quad TT^\dagger h_i(x)=2ih_i(x).
\]

\begin{lemma}\label{checkP}
Assume that the potential $V$ fulfills \eqref{potentialbound}. Then the perturbation matrix $\mathbb P(\theta)$ defined by \eqref{definitionP} is real analytic, mapping $\mathbb T^n_\sigma$ to $\M_{\wa}$ with $\alpha=\frac{1-\delta}2$, where $\delta\in(0,1)$ is the parameter appearing in \eqref{potentialbound}.
\end{lemma}

\begin{proof}
To simplify the notation, we denote by $\langle\cdot,\cdot\rangle$ the real scalar product and by the superscript $'$ the derivative w.r.t. $x$. Since $V$ itself is bounded, one obtains by Hölder's inequality and \eqref{potentialbound}
\[
||\bP||=||V(x,\cdot)||_{\B(L^2)}=\sup_{||f||=||g||=1}|\langle Vf,g\rangle|\le ||V(x,\cdot)||_\infty\le C.
\]
Now we are going to estimate the difference matrix of $\bP$. Since $\partial_x\circ V=V\circ\partial_x+V'$, then one has for all $i,j\ge1$
\begin{align*}
	&2(i-j)\bP_i^j(\theta)=\langle V\bH h_i,h_j\rangle-\langle Vh_i,\bH h_j\rangle=-\langle V\partial_xh'_i,h_j\rangle+\langle Vh_i(x),\partial_xh'_j\rangle\\
	=&\langle V'h'_i,h_j\rangle+\langle Vh'_i,h'_j\rangle-\langle Vh'_i,h'_j\rangle-\langle V'h_i,h'_j\rangle=\langle V'h'_i,h_j\rangle-\langle V'h_i,h'_j\rangle,
\end{align*}
which leads to
\[
2|i-j||\bP_i^j(\theta)|\le|\langle V'h'_i,h_j\rangle|+|\langle V'h_i,h'_j\rangle|.
\]
Then we estimate the terms above, one by one. Putting together Lemma \ref{kochLp} and \eqref{potentialbound}, \eqref{xDx}, we get 
\[
|\langle V'h'_i,h_j\rangle|\le\sqrt{i}\left(|\langle V'h_{i+1},h_j\rangle|+|\langle V'h_{i-1},h_j\rangle|\right)\le C \sqrt{i}(i\wedge j)^{\frac\delta2}.
\]
Similarly, one has $|\langle V'h_i,h'_j\rangle|\le C\sqrt{j}(i\wedge j)^\frac\delta2$. It follows that for all $i,j\ge1$
\begin{equation}\label{inejP}
	|i-j||\bP_i^j(\theta)|\le C\max\{i,j\}^\frac12(i\wedge j)^\frac\delta2,\quad\forall\,\theta\in\bT^n_\sigma.
\end{equation}
Now we are going to estimate $\Delta\bP_i^j(\theta)$. By definition \eqref{definitionP}, one gets
\stepcounter{equation}
\begin{align*}
\Delta\bP_i^j(\theta)&=\langle Vh_{i+1},h_{j+1}\rangle-\langle Vh_i,h_j\rangle
=\frac{1}{2\sqrt{ij}}\langle VT^\dagger h_i,T^\dagger h_j\rangle-\langle Vh_i,h_j\rangle\\
&=\frac{1}{2\sqrt{ij}}\langle TVT^\dagger h_i,h_j\rangle-\langle Vh_i,h_j\rangle\tag{\theequation}\label{TVT}\\
&=\frac1{2\sqrt{ij}}\langle VTT^\dagger h_i,h_j\rangle-\langle Vh_i,h_j\rangle+\frac1{2\sqrt{ij}}\langle V'T^\dagger h_i,h_j\rangle\\
&=\frac{(i-j)\bP_i^j}{\sqrt{j}(\sqrt{i}+\sqrt{j})}+\frac1{2\sqrt{ij}}\langle V'T^\dagger h_i,h_j\rangle.
\end{align*}
Combining \eqref{xDx}, \eqref{inejP} with Lemma \ref{kochLp}, one obtains for all $i,j\ge1$
\begin{equation}\label{sqrtj}
|\Delta\bP_i^j(\theta)|\le C\frac{(i\wedge j)^\frac\delta2}{\sqrt{j}}\le C(i\wedge j)^\frac{\delta-1}2,\quad\forall\,\theta\in\bT^n_\sigma.
\end{equation}
Letting $\alpha=\frac{1-\delta}2\in(0,1/2)$, 
the last estimate implies $|\Delta\bP(\theta)|_\alpha\le C$. Hence, we get for all $\theta\in\bT^n_\sigma$
\[
||\bP(\theta)||_{\wa}=\max\{||P(\theta)||,|\Delta \bP(\theta)|_\alpha\}\le C.
\]
This implies that $\bP(\theta)$ maps $\bT^n_\sigma$ to $\M_{\wa}$, from which the thesis follows.
\end{proof}

\begin{remark} 
In fact, $\Delta\bP_i^j$ could be estimated even better as follows. Observe that in \eqref{TVT} we have $\langle TVT^\dagger h_i,h_j\rangle=\langle h_i,TVT^\dagger h_j\rangle$. It follows
\begin{align*}
	|\Delta\bP_i^j(\theta)|&=\frac1{2\sqrt{ij}}\langle h_i,TVT^\dagger h_j\rangle-\langle Vh_i,h_j\rangle\\
	&=\frac1{2\sqrt{ij}}\langle h_i,VTT^\dagger h_j\rangle-\langle Vh_i,h_j\rangle+\frac1{2\sqrt{ij}}\langle h_i,V'T^\dagger h_j\rangle\\
	&=\frac{(j-i)\bP_i^j}{\sqrt{i}(\sqrt{i}+\sqrt{j})}+\frac1{2\sqrt{ij}}\langle h_i,V'T^\dagger h_j\rangle.
\end{align*}
Again by \eqref{xDx}, \eqref{inejP} and Lemma \ref{kochLp}  one has 
\[
|\Delta\bP_i^j(\theta)|\le C\frac{(i\wedge j)^\frac\delta2}{\sqrt{i}},\quad\forall\,\theta\in\bT^n_\sigma.
\]
Combining with \eqref{sqrtj}, we get
\[
|\Delta\bP_i^j(\theta)|\le C\frac{(i\wedge j)^\frac\delta2}{\max\{i,j\}^\frac12}\le C\max\{i,j\}^\frac{\delta-1}2.
\]
However, such a stronger decay would not persist in the KAM iteration. For this reason, we still use the weaker estimate \eqref{sqrtj}.
\end{remark}

\begin{remark}\label{logdecayP}
Assume that the potential $V$ is bounded and fulfills \eqref{logdecay} in Remark \ref{logDecay}. By an argument similar to the one above and using Lemma \ref{logdelta} one has for all $i,j\ge1$
\[
|\Delta\bP_i^j(\theta)|\le\frac{C}{\ln^{\delta'}(2+\max\{i,j\})}\le \frac{C}{\ln^{\delta'}(2+i\wedge j)},\quad\forall\,\theta\in\bT_\sigma^n.
\]
Then, by introducing suitable matrix spaces, one could prove the reducibility of system \eqref{initialSystem} and of the Schrödinger equation \eqref{scheqn} by the KAM method. 
\end{remark}

\subsection{Proof of Theorem \ref{mainThm}}
Letting $p\in[0,2]$, we are going to solve the Cauchy problem 
\begin{equation}\label{CauchySchInitial}
\begin{cases}
\mi\partial_t\psi=-\partial_x^2\psi+x^2\psi+\epsilon V(x,\omega t)\psi,\\
\psi(0,x)=\psi^{(0)}\in\cH^p.
\end{cases}
\end{equation}
Expanded in the Hermite basis, i.e. $\psi(t,x)=\sum_{i\ge1}u_i(t)h_i(x)$ and $\psi^{(0)}=\sum_{i\ge1}u^{(0)}_ih_i(x)$, the equation \eqref{CauchySchInitial} above is equivalent to 
\begin{equation}\label{CauchySystemInitial}
\begin{cases}
\mi\dot u=(\bA+\epsilon\bP(\omega t))u,\\
u(0)=u^{(0)}\in\ell_p^2,
\end{cases}
\end{equation}
where  $\bA=diag\{1,3,5,\dots\}$ and $\bP$ is defined by \eqref{definitionP}. Lemma \ref{checkP} gives  that $\bP\in\M_{\wa}(\Pi,\sigma)$, which allows us to apply Lemma \ref{reducibilityThm}. We therefore get a coordinate transformation $u=\U_{\omega,\epsilon}(\theta)v$ that conjugates above Cauchy problem \eqref{CauchySystemInitial} into an autonomous system diagonal
\begin{equation}\label{CauchySystemFinal}
	\begin{cases}
\mi\dot v=A^\infty v,\\
v(0)=\U^{-1}_{\omega,\epsilon}(0)u^{(0)}\in\ell_p^2,
\end{cases}
\end{equation}
where $A^\infty=diag\{\lambda_i^\infty\}_{i\in\bN}$. 
Clearly, the above Cauchy problem \eqref{CauchySystemFinal} is solvable. More precisely we have the unique solution $v(t)=\Exp{-\mi t A^\infty}v(0)$, i.e. the components $v_i(t)=\Exp{-\mi\lambda_i^\infty t}v_i(0)$ for $i\in\bN$. 
Lemma \ref{reducibilityThm} tells that  $u(t)=\U_{\omega,\epsilon}(\omega t)v(t)=\U_{\omega,\epsilon}(\omega t)\Exp{-\mi t A^\infty }\U_{\omega,\epsilon}^{-1}(0)u^{(0)}$ is the unique solution of Cauchy problem \eqref{CauchySystemInitial}.
Similarly, letting $\phi(t,x)=\sum_{i\ge1}v_i(t)h_i(x)$, the above system \eqref{CauchySystemFinal} is equivalent to 
\begin{equation}\label{CauchySchFinal}
\begin{cases}
\mi\partial_t\phi=H^\infty\phi,\\
\phi(0,x)=\sum_{i\ge1}v_i(0)h_i(x)\in\cH^p,
\end{cases}
\end{equation}
where $H^\infty=diag\{\lambda_i^\infty\}_{i\in\bN}$. 
Then according to the change of variable  $\U_{\omega,\epsilon}(\theta)$ acting on $\ell^2_p$, we define its associated coordinate transformation $\Psi_{\omega,\epsilon}(\theta)$ acting on $\cH^p$ by 
\[
\Psi_{\omega,\epsilon}(\theta) \Big(\sum_{i\ge1}v_ih_i(x)\Big)=\sum_{i\ge1}\Big(\U_{\omega,\epsilon}(\theta)v\Big)_ih_i(x),
\]
which directly conjugates the initial Cauchy problem \eqref{CauchySchInitial} to an autonomous system \eqref{CauchySchFinal} above. 
Equivalently, $\phi(t,\cdot)\in C(\bR,\cH^p)$ solves the equation  \eqref{CauchySchFinal} if and only if $\psi(t,\cdot)=\Psi_{\omega,\epsilon}(\omega t)\phi(t,\cdot)\in C(\bR,\cH^p)$ satisfies the initial equation  \eqref{CauchySchInitial}. More precisely, we have that the unique solution to the Cauchy problem \eqref{CauchySchInitial} is $\psi(t,x)=\Psi_{\omega,\epsilon}(\omega t) \Exp{-\mi tH^\infty} \Psi_{\omega,\epsilon}^{-1}(0) \psi^{(0)}$ and that estimates \eqref{reducibilityEstimate} imply \eqref{mainThmEstimate} since $\alpha=\frac{1-\delta}{2}$. \qed

\section{Proof of the reducibility theorem via KAM iteration}
As in \cite{liang22Non}, we are going to prove the reducibility of \eqref{initialSystem} by KAM tools. Now we briefly recall the general idea of the KAM iteration. We consider a non-autonomous system of the form
\[
\mi\dot u=(A+P)u,
\]
where $A$ is diagonal on the Hermite basis, time-independent and $\epsilon$-close to $\bA$, and the perturbation $P$ is of size $O(\epsilon)$ and quasi-periodically depending on time. We are about  to seek a suitable, time quasi-periodic, change of variable $u=\Exp{S} v$ transforming the above system into
\[
\mi\dot v=(A^++P^+)v,
\]
where $A^+$ will still be diagonal and $\epsilon$-close to $\bA$, while the new perturbation $P^+$ will become of size $O(\epsilon^2)$. More precisely, as in \cite{bambusi18I,bambusi01,liang22Non}, we have at least formally the identity 
\begin{gather*}
A^++P^+=A+([A,S]-\mi\dot S+P)\\
+\int_0^1\Exp{-\tau S}[(1-\tau)([A,S]-\mi\dot S+P)+\tau P,S]\Exp{\tau S}\md\tau,
\end{gather*}
where the second row above is of size $O(\epsilon^2)$. Hence, in order to achieve the goal we need to solve a homological equation of the form
\begin{equation}\label{homoeqn}
	[A,S]-\mi\dot S+P=A^+-A+R,
\end{equation}
where $R$ is an allowed error of order $O(\epsilon^2)$, which will be given by the high modes of the Fourier expansion of $P$ in time, and $\widetilde{A}:=A^+-A=O(\epsilon)$ will be given by the time average of the diagonal part of $P$ on the Hermite basis.

Then we iterate the above procedure, replacing $A$ with $A^+$, and the convergence of the scheme will allow us to build a change of variable $u=\U_{\omega,\epsilon}(\theta) v$ that transforms the original non-autonomous system \eqref{initialSystem} into an autonomous one of the form $\mi\dot v=A^\infty v$, with $A^\infty$ diagonal and $\epsilon$-close to $\bA$.

\subsection{Homological equation}
As explained above, we will solve a homological equation of the form \eqref{homoeqn} to find a suitable change of variable. Recall that, here and throughout the paper, we have $\alpha=\frac{1-\delta}2$, where $\delta\in(0,1)$ is the parameter appearing in \eqref{potentialbound}. The following result holds.

\begin{proposition}\label{homoSolution}
Let $\gamma\in(0,1/4)$. Assume that $\Pi\supset D\ni\omega\mapsto A(\omega)=diag\{\Lambda_1,\Lambda_2,\Lambda_3,\dots\}$ is a $C^1$ map satisfying, for some sufficiently small $\varepsilon\in(0,1-2\gamma)$,
\begin{equation}\label{homoVerify}
||A-\bA||_{\wa}^D\le\varepsilon.
\end{equation}
Let $P\in\M_{\wa}(D,\rho)$ be Hermitian with $\rho\in(0,\sigma]$, let $\kappa\in(0, \gamma]$ be sufficiently small, and $K\ge1$.
Then, setting $\nu_1=\frac{\alpha}{\alpha+2},\nu_2=n+1$, there is a subset $D'=D'(\kappa,K)\subset D$ satisfying
\begin{equation}\label{homoMeasure}
\meas(D\backslash D')\le C\kappa^{\nu_1}K^{\nu_2}
\end{equation}
and $C^1$ maps $\widetilde{A}:D'\mapsto\M_{\wa}$ diagonal and  $R:D'\times\bT^n_{\rho'}\mapsto\M_{\wa}$ Hermitian, $S:D'\times\bT^n_{\rho'}$ anti-Hermitian, all analytic in $\theta$, such that 
\begin{equation}\label{homoEqn}
[A,S]-\mi\dot S=\widetilde{A}-P+R,
\end{equation}
where $0<\rho'<\rho\le\sigma$. More precisely, the following estimates hold:
\begin{gather}
||\widetilde{A}||_{\wa}^{D'}\le ||P||_{\wa}^{D,\rho},\label{homoA}\\
||R||_{\wa}^{D',\rho'}\le\frac{C\Exp{-\frac K2(\rho-\rho')}}{(\rho-\rho')^n}||P||_{\wa}^{D,\rho},\label{homoR}\\
||S||_{\wa+}^{D',\rho'}\le\frac{CK^3}{\kappa^4(\rho-\rho')^n}||P||_{\wa}^{D,\rho}.\label{homoS}
\end{gather}
\end{proposition}

\begin{proof}
Let us first rewrite the homological equation \eqref{homoEqn} in Fourier expansion
\begin{equation}\label{homoFourier}
\mL\widehat{S}(k)=\delta_{k,0}\widetilde{A}-\widehat{P}(k)+\widehat{R}(k),
\end{equation}
where $\delta_{k,l}$ stands for the Kronecker symbol and $\mL:=\mL_{\omega,k}$ is a linear operator defined on $\M$ by 
\[
\mL B=k\cdot\omega B+[A(\omega),B].
\] 
Taking every matrix entry, \eqref{homoFourier} above reads 

\begin{equation}\label{homoEntry}
k\cdot\omega\widehat{S}_i^j(k)+(\Lambda_i-\Lambda_j)\widehat{S}_i^j(k)=\delta_{k,0}\widetilde{A}_i^j-\widehat{P}_i^j(k)+\widehat{R}_i^j(k),\quad k\in\bZ^n,i,j\in\bN.
\end{equation}

We first solve the trivial case where $|k|+|i-j|=0$ by setting
\[
\widehat{S}_i^i(0)=0,\quad\widehat{R}_i^i(0)=0,\quad\widetilde{A}_i^i=\widehat{P}_i^i(0).
\]
By letting $\widetilde{A}_i^j=0$ for $i\ne j$, one has that $\widetilde{A}\in\M_{\wa}$ satisfies $||\widetilde{A}||_{\wa}\le||\widehat{P}(0)||_{\wa}$. Differentiating the expression w.r.t $\omega$ one obtains the same type of estimate, from which \eqref{homoA} follows. 

Then we consider the remaining case where $|k|+|i-j|>0$. We solve equation \eqref{homoEntry} by setting for $i,j\ge1$ 
\begin{align}
\widehat{R}_i^j(k)&=
\begin{cases}
0,&\text{for }|k|\le K,\\
\widehat{P}_i^j(k),&\text{for }|k|>K;
\end{cases}\label{settingR}\\
\widehat{S}_i^j(k)&=
\begin{cases}
0,&\text{for }|k|>K\text{ or }|k|+|i-j|=0,\\
\frac{-\widehat{P}_i^j(k)}{k\cdot\omega+\Lambda_i-\Lambda_j}, &\text{otherwise}.
\end{cases}\label{settingS}
\end{align}
In such a way\footnote{Note that $R$ and $\widetilde{A}$ are Hermitian, but $S$ is anti-Hermitian because of the Hermitian  properties of $A$ and $P$.}, we keep the Fourier average ($k=0$) of the diagonal part ($i=j$) of $P$ as $\widetilde{A}$ and throw the high modes ($|k|>K$) of $P$ to the remainder $R$. Essentially, what we have solved is the following equation 
\[
k\cdot\omega\widehat{S}_i^j(k)+(\Lambda_i-\Lambda_j)\widehat{S}_i^j(k)=-\widehat{P}_i^j(k),\quad |k|\le K,~|k|+|i-j|>0.
\]
By definition \eqref{settingR}, Cauchy integral estimate gives that 
\[
|R(\theta)|_{\wa}\le\frac{C\Exp{-\frac K2(\rho-\rho')}}{(\rho-\rho')^n}\sup_{|\Im\theta|<\rho}||P(\theta)||_{\wa},\quad\forall\,|\Im\theta|<\rho'.
\]
Likewise, by differentiating w.r.t. $\omega$ one gets the same type of estimate, which leads to \eqref{homoR}. 

Let us now consider the solution $S$, where we are going to encounter small divisors. Thus, we have to introduce some parameters to gain proper estimates.  By definition, the norm $||S(\theta)||_{\wa+}$ consists of four parts as follows:
\[
||S(\theta)||,\ ||[\N,S(\theta)]||, \text{ and } |\Delta S(\theta)|_\alpha,\ |[\N,\Delta S(\theta)]|_{\alpha}.
\]
However, thanks to Cauchy integral estimate we only need to estimate their Fourier coefficients (w.r.t.\,$\theta$), i.e. 
\[
||\widehat{S}(k)||,~||[\N,\widehat{S}(k)]||,\text{ and }|\Delta\widehat{S}(k)|_{\alpha},~|[\N,\Delta\widehat{S}(k)]|_{\alpha}.
\]
We are going to estimate them one by one.

(1). Estimate of $||\widehat{S}(k)||$. In terms of small divisors, we distinguish two cases depending on whether $k=0$ or not.

(1a). The case $k=0$. In this case $i\ne j$. Applying the item (d) of Lemma \ref{Algebra} to \eqref{homoVerify} one has
\begin{equation}\label{lambdaij}
	|(\Lambda_i-\lambda_i)-(\Lambda_j-\lambda_j)|\le\frac{\varepsilon|i-j|}{(i\wedge j)^\alpha},\quad\forall\,i\ge1.
\end{equation}
It follows that  $|\Lambda_i-\Lambda_j|\ge|\lambda_i-\lambda_j|-\frac{\varepsilon|i-j|}{(i\wedge j)^\alpha}\ge\frac12(1+|i-j|)$, which implies $|\widehat{S}_i^j(0)|\le \frac{2|\widehat{P}_i^j(0)|}{1+|i-j|}$. Then  Lemma \ref{explanationM} shows that 
\(
||\widehat{S}(0)||\le C||\widehat{P}(0)||.
\)

(1b). The case $k\ne0$. In this case we have to face directly the small divisors. From  Lemma \ref{Melnikov}, we have a subset $\mathbb D=\mathbb D(\gamma,K)\subset\Pi$ such that for all $\omega\in\mathbb D$ the following holds:
\begin{equation}\label{initialDivisor}
|k\cdot\omega+\lambda_i-\lambda_j|\ge2 \gamma(1+|i-j|),\quad\forall\,i,j\in\bN \text{ and }\forall\,0\ne k\in\bZ^n,~|k|\le K.
\end{equation}
Observe that if $|i-j|\ge2\pi|k|$ then by \eqref{lambdaij} with $\varepsilon\le1-2\gamma$ one gets 
\[
|k\cdot\omega+\Lambda_i-\Lambda_j|\ge|k\cdot\omega+\lambda_i-\lambda_j|-\frac{\varepsilon|i-j|}{(i\wedge j)^\alpha}\ge(1-\varepsilon)|i-j|\ge\gamma(1+|i-j|).
\]
There are still no small divisors for $\omega\in\mathbb D$. Now we are going to consider the case $|i-j|\le2\pi|k|$ where the small divisors truly exist. Without loss of generality, assuming that $i\le j$, then we have, by \eqref{lambdaij}--\eqref{initialDivisor},
\stepcounter{equation}
\begin{align*}
	|k\cdot\omega+\Lambda_i-\Lambda_j|&\ge|k\cdot\omega+\lambda_i-\lambda_j|-\frac{\varepsilon|i-j|}{(i\wedge j)^\alpha}\\
	&\ge 2\gamma(1+|i-j|)-\frac{\varepsilon(1+|i-j|)}{i^\alpha}\\
	&\ge\gamma(1+|i-j|),\quad\underline{\text{provided } i\ge\left(\varepsilon/\gamma\right)^{1/\alpha}.}
	\tag{\theequation}\label{providedi}
\end{align*}
Define the set  
\[
F:=\bigcup_{\substack{i,j\in\bZ,|i-j|\le 2\pi|k|\\ k\in\bZ^n,0<|k|\le K}}\{\omega\in D:|k\cdot\omega+\Lambda_i-\Lambda_j|<\kappa(1+|i-j|)\}:=\bigcup_{\substack{i,j\in\bZ,|i-j|\le 2\pi|k|\\ k\in\bZ^n,0<|k|\le K}}F_{i,j}^k(\kappa)
\]
Since $\gamma\ge\kappa$, the last estimate \eqref{providedi} tells that $F_{i,j}^k(\kappa)=\emptyset$ when $j\ge i\ge(\varepsilon/\gamma)^{1/\alpha}$. Without loss of generality, assume that $|k_1|=\max_{i=1}^n\{|k_i|\}$. Then assumption \eqref{homoVerify} gives $|\partial_\omega\Lambda_i(\omega)|=|\partial_\omega(\Lambda_i-\lambda_i)|\le\varepsilon$, which implies 
\[
|\partial_{\omega_1}(k\cdot\omega+\Lambda_i-\Lambda_j)|\ge|k_1|-2\varepsilon\ge\frac{|k|}{2n}.
\]
Lemma \ref{measure} shows that $\meas\left(F_{i,j}^k(\kappa)\right)\le\frac{C(n)\kappa(1+|i-j|)}{|k|}$. Now we calculate
\[
\meas(F)\le\sum_{\substack{i,j\in\bZ,|i-j|\le 2\pi|k|\\ k\in\bZ^n,0<|k|\le K}}\meas\left(F_{i,j}^k(\kappa)\right)\le \sum_{\substack{i,j\in\bZ\\(i\wedge j)\le(\varepsilon/\gamma)^{1/\alpha},|i-j|\le2\pi|k|\\ k\in\bZ^n,0<|k|\le  K}}C\kappa\le C\kappa\gamma^{-2/\alpha}K^{n+1}.
\]
Letting $D'=\mathbb D\cap(D\backslash F)$, the previous estimate leads to 
\[
\meas(D\backslash D')\le\meas(\Pi\backslash\mathbb D)+\meas(F)\le C\gamma K^{n+1}+C\kappa\gamma^{-2/\alpha}K^{n+1}.
\]
Setting $\kappa=\gamma^{1+2/\alpha}$, one has 
\[
\meas(D\backslash D')\le \kappa^{\frac{\alpha}{\alpha+2}}K^{n+1}. 
\]
Combining all the above estimates on small divisors, one has for all $\omega\in D'$
\begin{equation}\label{newDivisor}
	|k\cdot\omega+\Lambda_i-\Lambda_j|\ge \kappa(1+|i-j|),\quad\forall\,i,j\in\bN \text{ and }\forall\,0\ne k\in\bZ^n,~|k|\le K.
\end{equation}
Namely,  the subset $D'=D'(\kappa,K)$ is our desired one and the above estimate concludes \eqref{homoMeasure} with $\nu_1=\frac{\alpha}{\alpha+2},\nu_2=n+1$.
Also, 
by setting \eqref{settingS} it follows  that $|\widehat{S}_i^j(k)|\le\frac{|\widehat{P}_i^j(k)|}{\kappa(1+|i-j|)}$, which, together with Lemma \ref{explanationM}, implies that
\begin{equation}\label{estimateS1}
	||\widehat{S}(k)||\le\frac{C}{\kappa}||\widehat{P}(k)||,\quad\forall~k\in\bZ^n,~|k|\le K.
\end{equation}

(2). Estimate of $||[\N,\widehat{S}(k)]||$. We have proved that for all $\omega\in D'$ equation \eqref{homoFourier} has a solution $\widehat{S}(k)\in\M$ for any $k\in\bZ^n,~|k|\le K$. At the same time, we have 
\[
[\bA,\widehat{S}(k)]=[(\bA-A),\widehat{S}(k)]+k\cdot\omega\widehat{S}(k)+\delta_{k,0}\widetilde{A}-\chi_{|k|\le K}(k)\cdot\widehat{P}(k),
\]
which implies by \eqref{homoVerify} that 
\begin{equation}\label{estimateS2}
||[\N,\widehat{S}(k)]||=\frac12
||[\bA,\widehat{S}(k)]||\le\frac{CK}{\kappa}||\widehat{P}(k)||,\quad\forall\,k\in\bZ^n,~|k|\le K.
\end{equation}

(3). Estimate of $|\Delta\widehat{S}(k)|_{\alpha+}=\max\{|\Delta \widehat{S}(k)|_{\alpha},|[\N,\Delta\widehat{S}(k)]|_\alpha\}$. Thanks to Remark \ref{remarkMatrix}\,(2) we can estimate the last two norms together. By the definition of $\widehat S$, we directly compute 
\begin{align*}
	\Delta\widehat{S}_i^j(k)&=\widehat{S}_{i+1}^{j+1}(k)-\widehat{S}_i^j(k)=\frac{-\widehat{P}_{i+1}^{j+1}(k)}{k\cdot\omega+\Lambda_{i+1}-\Lambda_{j+1}}-\frac{-\widehat{P}_i^j(k)}{k\cdot\omega+\Lambda_i-\Lambda_j}\\
	&=\frac{-\Delta\widehat{P}_i^j(k)}{k\cdot\omega+\Lambda_{i+1}-\Lambda_{j+1}}+\frac{(\Lambda_{i+1}-\Lambda_i-\Lambda_{j+1}+\Lambda_j)\widehat{P}_i^j(k)}{(k\cdot\omega+\Lambda_{i+1}-\Lambda_{j+1})(k\cdot\omega+\Lambda_i-\Lambda_j)}\\
	&=\frac{-\Delta\widehat{P}_i^j(k)}{k\cdot\omega+\Lambda_{i+1}-\Lambda_{j+1}}+\frac{\big(\Delta(A-\bA)_i^i-\Delta(A-\bA)_j^j)\big)\widehat{P}_i^j(k)}{(k\cdot\omega+\Lambda_{i+1}-\Lambda_{j+1})(k\cdot\omega+\Lambda_i-\Lambda_j)}.
\end{align*}
It follows by \eqref{homoVerify}, \eqref{newDivisor} and Lemma \ref{Algebra}\,$(d)$ that
\[
|\Delta\widehat{S}_i^j(k)|\le\frac{|\Delta\widehat{P}(k)|_\alpha}{\kappa(1+|i-j|)(i\wedge j)^\alpha}+\frac{\varepsilon||\widehat{P}(k)||}{\kappa^2(1+|i-j|)(i\wedge j)^\alpha}\le\frac{C||\widehat{P}(k)||_{\wa}}{\kappa^2(1+|i-j|)(i\wedge j)^\alpha},
\]
which implies 
\begin{equation}\label{estimateS3}
	|\Delta\widehat{S}(k)|_{\alpha+}\le\frac C{\kappa^2}||\widehat{P}(k)||_{\wa},\quad\forall\,k\in\bZ^n,~|k|\le K.
\end{equation}
Collecting estimates \eqref{estimateS1}--\eqref{estimateS3}, we obtain 
\begin{equation}\label{estimateSk}
||\widehat{S}(k)||_{\wa+}\le\frac{CK}{\kappa^2}||\widehat{P}(k)||_{\wa},\quad\forall\,k\in\bZ^n,~|k|\le K,
\end{equation}
from which it follows by Cauchy's estimate that 
\begin{equation}\label{estimateStheta}
||S(\theta)||_{\wa+}\le\frac{CK}{\kappa^2(\rho-\rho')^n}\sup_{|\Im\theta|<\rho}||P(\theta)||_{\wa},\quad\forall\,|\Im\theta|<\rho'.
\end{equation}
By setting \eqref{settingR}, the homological equation \eqref{homoFourier} reads 
\begin{equation}\label{homoFourierK}
\mL\widehat{S}(\omega,k)=\delta_{k,0}\widetilde{A}(\omega)-\widehat{P}(\omega,k)+\chi_{|k|>K}(k)\cdot\widehat{P}(\omega,k).
\end{equation}
Differentiating equation \eqref{homoFourierK} gives
\begin{equation}\label{homoFourierK.om}
\mL\partial_\omega\widehat{S}(\omega,k)=\delta_{k,0}\partial_\omega\widetilde{A}(\omega)-\Big(\partial_\omega\widehat{P}(\omega,k)+(\partial_\omega\mL)\widehat{S}(\omega,k)\Big)+\chi_{|k|>K}(k)\cdot\partial_\omega\widehat{P}(\omega,k).
\end{equation}
Let $Q(\omega,k)=\partial_\omega\widehat{P}(\omega,k)+(\partial_\omega\mL)\widehat{S}(\omega,k)$, then \eqref{settingS} implies $\chi_{|k|>K}(k)\cdot Q(\omega,k)=\chi_{|k|>K}(k)\cdot\partial_\omega\widehat{P}(\omega,k)$. It follows that \eqref{homoFourierK.om} reads 
\begin{equation}\label{homoFourierKp}
	\mL\partial_\omega\widehat{S}(\omega,k)=\delta_{k,0}\partial_\omega\widetilde{A}(\omega)-Q(\omega,k)+\chi_{|k|>K}(k)\cdot Q(\omega,k),
\end{equation}
which is formally the same as equation \eqref{homoFourierK} with $\widehat{P}(\omega,k)$ replaced by $Q(\omega,k)$. We can now solve equation \eqref{homoFourierKp} by defining 
\[
\partial_\omega\widehat{S}(\omega,k)=\chi_{|k|\le K}(k)\cdot\mL^{-1}_{\omega,k}
\Big(\delta_{k,0}\partial_\omega\widetilde{A}(\omega)-Q(\omega,k)+\chi_{|k|>K}(k)\cdot Q(\omega,k)\Big).
\]
Recalling assumption \eqref{homoVerify} and estimate \eqref{estimateSk}, we have for all $k\in\bZ^n,~|k|\le K$
\[
||Q(\omega,k)||_{\wa}\le||\partial_\omega\widehat{P}(\omega,k)||_{\wa}+CK||\widehat{S}(\omega,k)||_{\wa+}\le\frac{CK^2}{\kappa^2}
\max\{||\widehat{P}(\omega,k)||_{\wa},||\partial_\omega\widehat{P}(\omega,k)||_{\wa}\}.
\]
Similarly to \eqref{estimateSk}, 
repeating almost the same procedures as before gives for all $k\in\bZ^n,~|k|\le K$
\begin{align*}
||\partial_\omega\widehat{S}(\omega,k)||_{\wa+}\le\frac{CK}{\kappa^2}||Q(\omega,k)||_{\wa}\le\frac{CK^3}{\kappa^4}
\max\{||\widehat{P}(\omega,k)||_{\wa},||\partial_\omega\widehat{P}(\omega,k)||_{\wa}\},
\end{align*}
which implies by Cauchy's estimate that
\[
||\partial_\omega S(\omega,\theta)||_{\wa+}\le\frac{CK^3}{\kappa^4(\rho-\rho')^n}\sup_{|\Im\theta|<\rho}
\max\{||P(\omega,\theta)||_{\wa},||\partial_\omega P(\omega,\theta)||_{\wa}\},
\quad\forall\,|\Im\theta|<\rho'.
\]
Together with \eqref{estimateStheta} we get the expected estimate \eqref{homoS}. The proof is now complete.

\end{proof}

\subsection{KAM iteration}
In this section we will perform our KAM scheme as follows. Let us denote by $m\ge0$ the index of the current step of the KAM iteration and initialize our non-autonomous system (i.e. $m=0$) 
\begin{equation}\label{initialStep}
\mi\dot u=(A_0+P_0(\omega t))u,
\end{equation}
where $A_0=\bA=diag\{\lambda_i\}_{i\ge1}$ with $\lambda_i=2i-1$ and $P_0=\epsilon\bP\in\M_{\wa}(\Pi,\sigma)$. 
Assuming that the previous $m$ steps have already been done, at the $m^\mth$ step one gets an equation of the form
\begin{equation}\label{mStep}
	\mi\dot u=(A_m+P_m)u,
\end{equation}
where $A_m=diag\{\lambda_i^{(m)}\}_{i\ge1}$ and $P_m\in\M_{\wa}(\Pi_m,\sigma_m)$. We are now going to build a change of variable $u=\Exp{S_{m+1}}v$ defined on a subset $\Pi_{m+1}\times\bT^n_{\sigma_{m+1}}\subset\Pi_m\times\bT^n_{\sigma_m}$, 
transforming equation \eqref{mStep} into a new equation at the $(m+1)^\mth$ step of the form
\begin{equation}\label{m+Step}
\mi\dot v=(A_{m+1}+P_{m+1})v,
\end{equation}
where $A_{m+1}=diag\{\lambda_i^{(m+1)}\}_{i\ge1}$ and $P_{m+1}\in\M_{\wa}(\Pi_{m+1},\sigma_{m+1})$.
More precisely, we first use Proposition \ref{homoSolution} to construct $S_{m+1}$ by solving the homological equation
\begin{equation}\label{KAMhomo}
[A_m,S_{m+1}]-\mi\dot S_{m+1}=\widetilde{A}_m-P_m+R_m,\quad(\omega,\theta)\in\Pi_{m+1}\times\bT^n_{\sigma_{m+1}},
\end{equation}
where $\widetilde{A}_m(\omega)$ and $R_m(\omega,\theta)$ are defined respectively on $\Pi_{m+1}$ and $\Pi_{m+1}\times\bT^n_{\sigma_{m+1}}$ by
\begin{align}
\widetilde{A}_m(\omega)&=\left(\delta_{i,j}\left(\widehat{P}(\omega,0)\right)_i^j\right)_{i,j\ge1},\label{KAMa}\\
R_m(\omega,\theta)&=\sum_{|k|>K_{m+1}}\widehat{P}(\omega,k)\Exp{\mi k\cdot\theta}.\label{KAMR}
\end{align}
Then by the coordinate transformation $u=\Exp{S_{m+1}}v$ we get the new equation \eqref{m+Step}, where $A_{m+1}$ and $P_{m+1}$ are also defined respectively on $\Pi_{m+1}$ and $\Pi_{m+1}\times\bT^n_{\sigma_{m+1}}$ by
\begin{align}
A_{m+1}&=A_m+\widetilde{A}_m,\label{KAMA}
\\
P_{m+1}&=R_m+\int_0^1\Exp{-\tau S_{m+1}}[(1-\tau)(\widetilde{A}_m+R_m)+\tau P_m,S_{m+1}]\Exp{\tau S_{m+1}}\md\tau.\label{KAMP}
\end{align}
Observe that by construction if $A_m$ and $P_m$ are both Hermitian, then so are $\widetilde{A}_m,R_m$ and $A_{m+1}$. It follows from the solution to the homological equation \eqref{KAMhomo} that $S_{m+1}$ is anti-Hermitian, which implies that $P_{m+1}$ is again Hermitian. As can be easily seen, this structure is preserved along the KAM iteration. By iterating the procedure \eqref{mStep}-\eqref{KAMP} we build a change of variable $u=\U_{m+1}v:=\Exp{S_1}\circ \Exp{S_2}\circ\dots\circ \Exp{S_{m+1}}v$ transforming the initial equation \eqref{initialStep} into the new equation \eqref{m+Step} at the $(m+1)^\mth$ step. Before setting the parameters of the iteration, note that $\lambda_i^{(0)}=\lambda_i$ for $i\in\bN$,  $\Pi_0=\Pi$, $\sigma_0=\sigma$ and $S_0=\zero,\U_0=\id$. Let $||P_0||_{\wa}^{\Pi_0,\sigma_0}\le\epsilon_0$, with $\epsilon_0>0$ small enough. Then choose, for $m\ge0$,
\begin{gather*}
\epsilon_{m+1}=\epsilon_{m}^{4/3},\quad\kappa_{m+1}=\epsilon_{m}^{1/16},\\
\sigma_{m}-\sigma_{m+1}=\frac{\sigma_0{(m+1)}^{-2}}{2\sum_{i\ge1}i^{-2}},\\
K_{m+1}=\frac{2\ln\epsilon_{m}^{-1}}{\sigma_{m}-\sigma_{m+1}}.
\end{gather*}

\begin{lemma}[KAM Iteration]\label{KAMiteration}
Let $\alpha\in(0,1],\nu_1=\frac{\alpha}{\alpha+2}$ and $m\ge0$. There exists $\epsilon_*=\epsilon_*(\sigma,n,\alpha)\ll1$ such that for all $\epsilon\in(0,\epsilon_*)$ there are $\Pi_{m+1}\subset\Pi_{m},\sigma_{m+1}<\sigma_{m}$ and $S_{m+1}\in\M_{\wa+}(\Pi_{m+1},\sigma_{m+1}),P_{m+1}\in\M_{\wa}(\Pi_{m+1},\sigma_{m+1})$ such that the change of variable $u=\Exp{S_{m+1}(\omega,\theta)}v$, defined on $\Pi_{m+1}\times\bT^n_{\sigma_{m+1}}$ and acting from $\ell^2_p$ into itself, is a unitary (on $\ell^2$) isomorphism, which conjugates the system $\mi\dot u=(A_{m}+P_{m})u$ at the $m^\mth$ step to the system $\mi\dot v=(A_{m+1}+P_{m+1})v$ at the $(m+1)^\mth$ step. Moreover, the following estimates hold:
\begin{gather}
	\meas(\Pi_{m}\backslash\Pi_{m+1})\le\epsilon_{m}^{\frac{\nu_1}{17}},\label{mKAMM}\\
	||\widetilde{A}_{m}||_{\wa}^{\Pi_{m+1}}\le\epsilon_{m},\label{mKAMA}\\
	||P_{m+1}||_{\wa}^{\Pi_{m+1},\sigma_{m+1}}\le\epsilon_{m+1},\label{mKAMP}\\
		||S_{m+1}||_{\wa+}^{\Pi_{m+1},\sigma_{m+1}}\le\epsilon_{m}^{2/3},\notag\\
	||\U_{m+1}(\omega,\theta)-\id||_{\B(\ell^2_p)}\le\sum_{l=0}^m2\epsilon_{l}^{2/3},\quad\forall\,(\omega,\theta)\in\Pi_{m+1}\times\bT^n_{\sigma_{m+1}},\label{mKAMU}
\end{gather}
where $p\in[-2,2]$ and  $\U_{m+1}=\Exp{S_1}\circ\Exp{S_2}\circ\dots\circ\Exp{S_{m+1}}$(in particular $\U_0=\id$).
\end{lemma}
\begin{proof}
We are going to proceed by introduction using Proposition \ref{homoSolution}. Initially we have $A_0=\bA$, which verifies \eqref{homoVerify}, then by Proposition \ref{homoSolution} one constructs $\Pi_1,\sigma_1$ and $\widetilde{A}_0,R_0,S_1$ such that the following homological equation holds on $\Pi_1\times\bT^n_{\sigma_1}$
\[
[A_0,S_1]-\mi\dot S_1=\widetilde{A}_0-P_0+R_0.
\]
Recalling measure estimate \eqref{homoMeasure}, we get
\[
\meas(\Pi_0\backslash\Pi_1)\le C\kappa_1^{\nu_1}K_1^{\nu_2}\le C(\sigma,n)\epsilon_0^{\frac{\nu_1}{16}}\left(\ln\epsilon_0^{-1}\right)^{\nu_2}\le\epsilon_0^{\frac{\nu_1}{17}}.
\]
 Due to \eqref{homoS}  one has 
\[
||S_1||_{\wa+}^{\Pi_1,\sigma_1}\le\frac{CK_1^3}{\kappa_1^4(\sigma_0-\sigma_1)^n}||P_0||_{\wa}^{\Pi_0,\sigma_0}\le C(\sigma,n)\epsilon_0^{3/4}\left(\ln\epsilon_0^{-1}\right)^3\le\epsilon_0^{2/3}.
\]
It follows by Lemma \ref{Algebra} that for all $(\omega,\theta)\in\Pi_1\times\bT^n_{\sigma_1}$ 
\begin{align*}
	&||\U_1(\omega,\theta)-\id||_{\B(\ell_p^2)}=||\Exp{S_1(\omega,\theta)}-\id||_{\B(\ell_p^2)}\\
	\le& C\Exp{C||S_1||_{\wa+}^{\Pi_1,\sigma_1}}||S_1||_{\wa+}^{\Pi_1,\sigma_1}\le2\epsilon_0^{2/3},\qquad\forall\,p\in[-2,2].
\end{align*}
Collecting all the estimates \eqref{homoA}-\eqref{homoS} we get
$||\widetilde{A}_0||_{\wa}^{\Pi_1}\le\epsilon_0$ and 
\[
||R_0||_{\wa}^{\Pi_1,\sigma_1}\le\frac{C\Exp{-\frac{K_1}2(\sigma_0-\sigma_1)}}{(\sigma_0-\sigma_1)^n}||P_0||_{\wa}^{\Pi_0,\sigma_0}\le C(\sigma,n)\epsilon_0^2\le\frac12\epsilon_0^{4/3}.
\]  
In addition, Lemma \ref{Algebra} gives that for all $\tau\in[0,1]$
\[
||[(1-\tau)(\widetilde{A}_0+R_0)+\tau P_0,S_1]||_{\wa}^{\Pi_1,\sigma_1}\le C||P_0||_{\wa}^{\Pi_0,\sigma_0}||S_1||_{\wa+}^{\Pi_1,\sigma_1}\le C\epsilon_0^{5/3},
\]
which implies that 
\[
\left\|\int_0^1\Exp{-\tau S_1}[(1-\tau)(\widetilde{A}_0+R_0)+\tau P_0,S_1]\Exp{\tau S_1}\md\tau\right\|_{\wa}^{\Pi_1,\sigma_1}\le\frac12\epsilon_0^{4/3}.
\]
Together with the estimate on $R_0$ above, definition \eqref{KAMP} shows  $||P_1||_{\wa}^{\Pi_1,\sigma_1}\le\epsilon_0^{4/3}=\epsilon_1$.

Then assuming the precedent $m$ steps have already done, we are going from the $m^\mth$ step to the $(m+1)^\mth$ step. Clearly, by \eqref{mKAMA} it follows that 
\[
||A_{m}-\bA||_{\wa}^{\Pi_{m}}=\left\|\sum_{l=0}^{m-1}\widetilde{A}_l\right\|_{\wa}^{\Pi_{m}}\le\sum_{l=0}^{m-1}||\widetilde{A}_l||_{\wa}^{\Pi_{l+1}}\le\sum_{l=0}^{m-1}\epsilon_l\le 2\epsilon_0,
\]
so that assumption \eqref{homoVerify} is verified provided that $2\epsilon_0\leq\varepsilon$, where $\varepsilon$ is the small parameter appearing in Proposition \ref{homoSolution}. Similarly as before, we apply the Proposition \ref{homoSolution} again to construct $\Pi_{m+1},\sigma_{m+1}$ and $\widetilde{A}_{m},R_m,S_{m+1}$ such that the following homological equation holds on $\Pi_{m+1}\times\bT^n_{\sigma_{m+1}}$:
\[
[A_{m},S_{m+1}]-\mi\dot S_{m+1}=\widetilde{A}_m-P_{m}+R_m.
\]
Recalling measure estimate \eqref{homoMeasure}, we conclude that 
\[
\meas(\Pi_{m}\backslash\Pi_{m+1})\le C\kappa_{m+1}^{\nu_1}K_{m+1}^{\nu_2}\le C(\sigma,n)\epsilon_{m}^{\frac{\nu_1}{16}}\left(m^2\ln\epsilon_{m}^{-1}\right)^{\nu_2}\le\epsilon_{m}^{\frac{\nu_1}{17}}.
\]
Due to \eqref{homoS} we get 
\[
||S_{m+1}||_{\wa+}^{\Pi_{m+1},\sigma_{m+1}}\le\frac{CK_{m+1}^3}{\kappa_{m+1}^4(\sigma_{m}-\sigma_{m+1})^n}||P_m||_{\wa}^{\Pi_{m},\sigma_{m}}\le C(\sigma,n)\epsilon_{m}^{3/4}m^{2(n+3)}\left(\ln\epsilon_{m}^{-1}\right)^3\le\epsilon_{m}^{2/3}.
\]
Furthermore, by Lemma \ref{Algebra} and estimate \eqref{mKAMU}, one gets for all $(\omega,\theta)\in\Pi_{m+1}\times\bT^n_{m+1}$
\begin{gather}
||\Exp{S_{m+1}(\omega,\theta)}-\id||_{\B(\ell^2_p)}\le C\Exp{C||S_{m+1}||_{\wa+}^{\Pi_{m+1},\sigma_{m+1}}}||S_{m+1}||_{\wa+}^{\Pi_{m+1},\sigma_{m+1}}\le\epsilon_{m}^{2/3},\label{mKAMeS}\\
||\U_m(\omega,\theta )-\id||_{\B(\ell_p^2)}\le\sum_{l=0}^{m-1}2\epsilon_{l}^{2/3}\le 4\epsilon_0.\notag
\end{gather}
Since, by definition,
\(
\U_{m+1}-\id=\U_m\circ\Exp{S_{m+1}}-\id=\U_m\circ(\Exp{S_{m+1}}-\id)+\U_m-\id
\), then we have 
\[
||\U_{m+1}-\id||_{\B(\ell^ 2_p)}\le ||\U_m||_{\B(\ell_p^ 2)}\cdot||\Exp{S_{m+1}}-\id||_{\B(\ell_p^ 2)}+||\U_m-\id||_{\B(\ell_p^ 2)}\le\sum_{l=0}^{m}2\epsilon_{l}^{2/3}.
\]
Combining estimates \eqref{homoA}-\eqref{homoS} we conclude that $||\widetilde{A}_{m}||_{\wa}^{\Pi_{m+1}}\le\epsilon_{m}$ and 
\begin{equation}\label{estimate.Rm}
||R_m||_{\wa}^{\Pi_{m+1},\sigma_{m+1}}\le\frac{C\Exp{-\frac{K_{m+1}}{2}(\sigma_{m}-\sigma_{m+1})}}{(\sigma_{m}-\sigma_{m+1})^n}||P_m||_{\wa}^{\Pi_{m},\sigma_{m}}\le C(\sigma,n)\epsilon_{m}^2m^{2n}\le\frac12\epsilon_{m}^{4/3}.
\end{equation}
Finally, Lemma \ref{Algebra} shows that  for all $\tau\in[0,1]$
\[
||[(1-\tau)(\widetilde{A}_m+R_m)+\tau P_m,S_{m+1}]||_{\wa}^{\Pi_{m+1},\sigma_{m+1}}\le C||P_m||_{\wa}^{\Pi_{m},\sigma_{m}}\cdot||S_{m+1}||_{\wa+}^{\Pi_{m+1},\sigma_{m+1}}\le C\epsilon_{m}^{5/3},
\]
from which it follows that 
\[
\left\|\int_0^1\Exp{-\tau S_{m+1}}[(1-\tau)(\widetilde{A}_m+R_m)+\tau P_m,S_{m+1}]\Exp{\tau S_{m+1}}\md\tau\right\|_{\wa}^{\Pi_{m+1},\sigma_{m+1}}\le\frac12\epsilon_{m}^{4/3}.
\]
By \eqref{estimate.Rm} and \eqref{KAMP} we have  $||P_{m+1}||_{\wa}^{\Pi_{m+1},\sigma_{m+1}}\le\epsilon_{m}^{4/3}=\epsilon_{m+1}$, which completes the proof.
\end{proof}

\subsection{Proof of Theorem \ref{reducibilityThm}}
We are now going to prove the reducibility theorem using the KAM Iteration Lemma \ref{KAMiteration} above. Defining $\Pi_\epsilon=\bigcap_{m\ge0}\Pi_{m}$, the measure estimates \eqref{mKAMM} give 
\[
\meas(\Pi\backslash\Pi_\epsilon)\le\sum_{m\ge0}\meas(\Pi_{m}\backslash\Pi_{m+1})\le\sum_{m\ge0}\epsilon_{m}^{\frac{\nu_1}{17}}\le2\epsilon_0^\frac{\nu_1}{17},\quad\nu_1=\frac{\alpha}{\alpha+2}.
\]
By construction, we have $\sigma_\infty:=\sigma-\sum_{m\ge0}(\sigma_{m}-\sigma_{m+1})=\sigma/2$. In the following, let $(\omega,\theta)\in\Pi_\epsilon\times\bT^n_{\sigma/2}$ and $p\in[-2,2]$. Clearly, by \eqref{mKAMP}, $P_m$ goes to 0 as $m$ goes to infinity. Writing $A^\infty=diag\{\lambda_i^\infty\}_{i\in\bN}$ with $\lambda_i^\infty=\lim_{m\to\infty}\lambda_i^{(m)}$ for $i\in\bN$, in view of \eqref{mKAMA}, we get 
\[
||A^\infty-A_0||_{\wa}^{\Pi_\epsilon}\le\sum_{m\ge0}||\widetilde{A}_m||_{\wa}^{\Pi_{m+1}}\le\sum_{m\ge0}\epsilon_{m}\le2\epsilon_0,
\]
which implies that $|\lambda_i^\infty-\lambda_i|\le2\epsilon_0$ for all $i\in\bN$. We have now proved the first two estimates of \eqref{reducibilityEstimate}. Then we are going to estimate the coordinate transformation $\U_{\omega,\epsilon}(\theta)$ that appears in Theorem \ref{reducibilityThm}. Since the definition shows that $\U_{m+1}-\U_m=\U_m\circ(\Exp{S_{m+1}}-\id)$, we get from estimates \eqref{mKAMU} and \eqref{mKAMeS} that 
\[
||\U_{m+1}-\U_m||_{\B(\ell_p^2)}\le||\U_m||_{\B(\ell_p^2)}\cdot||\Exp{S_{m+1}}-\id||_{\B(\ell_p^2)}\le2\epsilon_{m}^{2/3}.
\]
It follows that for any $m_2>m_1\ge0$
\[
||\U_{m_2}-\U_{m_1}||_{\B(\ell_p^2)}\le\sum_{l=m_1}^{m_2-1}||\U_{l+1}-\U_l||_{\B(\ell_p^2)}\le\sum_{l\ge m_1}2\epsilon_{l}^{2/3}\le4\epsilon_{m_1}^{2/3}\to0,\text{ as }m_1\to\infty,
\]
which implies that $\{\U_m(\omega,\theta)\}_{m\ge0}$ is a Cauchy sequence (uniformly in $\omega$ and $\theta$) in $\B(\ell_p^2)$. Denoting by $\U_{\omega,\epsilon}(\theta)$ its limiting map, then \eqref{mKAMU} gives the last estimate of \eqref{reducibilityEstimate} in Theorem \ref{reducibilityThm} and the uniform convergence implies the $C^1$ regularity (in $\omega$) and analyticity (in $\theta$).

Let us spend some more words on why we get the reducibility in $\ell^2$. As shown in Lemma \ref{KAMiteration}, for all $m\ge0$  the change of variable $\ell^2\ni u=\U_m v\in\ell^2$ conjugates the initial equation \eqref{initialStep} to the equation $\mi\dot v=(A_{m}+P_{m})v$ at the $m^\mth$ step, where the identity $A_m+P_m\equiv\U_m^{-1}(\bA+\epsilon\bP)\U_m-\mi\,\U_m^{-1}\partial_t\,\U_m$ holds not only formally but also truly in $\B(\ell^2,\ell_{{-2}}^2)$ by Lemma \ref{Algebra}\,$(c)$. The proof is now complete. \qed

\appendix

\section{Proof of Lemma \ref{Algebra}}\label{proofAlgebra}
\begin{proof}
	
$(a)$. Recall that $A\in\M_{\wa+},B\in\M_{\wa}$ and
\begin{gather*}
	||A||_{\wa+}=\max\{||A||,||[\N,A]||,|\Delta A|_\alpha,|[\N,\Delta A]|_\alpha\},\\
	||B||_{\wa}=\max\{||B||,|\Delta B|_\alpha\}.
\end{gather*}
Then we need to estimate $||AB||_{\wa}=\max\{||AB||,|\Delta (AB)|_\alpha\}$.
First we have $||AB||\le||A||\cdot||B||$. Now we compute for any $i,j\ge1$

\begin{align}
	\Delta(AB)_i^j&=(AB)_{i+1}^{j+1}-(AB)_i^j=\sum_{k\ge1}A_{i+1}^kB_k^{j+1}-\sum_{k\ge1}A_i^kB_k^j\notag\\
	&=A_{i+1}^1B_1^{j+1}+\sum_{k\ge1}A_{i+1}^{k+1}B_{k+1}^{j+1}-\sum_{k\ge1}A_i^kB_k^j\notag\\
	&=A_{i+1}^1B_1^{j+1}+\sum_{k\ge1}\Delta A_i^k\cdot B_{k+1}^{j+1}+\sum_{k\ge1}A_i^k\cdot\Delta B_k^j.
	\label{differenceProdoct}
\end{align} 
It follows that 
\[
|\Delta (AB)_i^j|\le|A_{i+1}^1|\cdot|B_1^{j+1}|+\sum_{k\ge1}|\Delta A_i^k|\cdot |B_{k+1}^{j+1}|+\sum_{k\ge1}|A_i^k|\cdot|\Delta B_k^j|:=\Delta_1+\Delta_2+\Delta_3.
\]
Then we estimate the above summands one by one. Since $\alpha\in[0,1]$ we have 
\[
\Delta_1\le \frac{||[\N,A]||}{i}\cdot ||B||\le\frac{||[\N,A]||\cdot||B||}{(i\wedge j)^\alpha}.
\]
Plus, we get 
\begin{align*}
\Delta_2&\le\sum_{k\ge1}\frac{|\Delta A|_\alpha+|[\N,\Delta A]|_\alpha}{(i\wedge k)^\alpha(1+|i-k|)}\cdot|B_{k+1}^{j+1}|=
\left(|\Delta A|_\alpha+|[\N,\Delta A]|_\alpha\right)\sum_{k\ge1}\frac{|B_{k+1}^{j+1}|}{(i\wedge k)^\alpha(1+|i-k|)}\\
&\le \left(|\Delta A|_\alpha+|[\N,\Delta A]|_\alpha\right)\left(\sum_{k\ge i/2}+\sum_{k\le i/2}\right)\frac{|B_{k+1}^{j+1}|}{(i\wedge k)^\alpha(1+|i-k|)}\\
&\le \left(|\Delta A|_\alpha+|[\N,\Delta A]|_\alpha\right)\left(\sum_{k\ge i/2}\frac{2|B_{k+1}^{j+1}|}{i^\alpha(1+|i-k|)}+\sum_{k\le i/2}\frac{2|B_{k+1}^{j+1}|}{i^\alpha k}\right)\\
&\le\frac{2\left(|\Delta A|_\alpha+|[\N,\Delta A]|_\alpha\right)}{i^\alpha}
\left(\sum_{k\ge1}\frac{|B_{k+1}^{j+1}|}{1+|i-k|}+\sum_{k\ge1}\frac{|B_{k+1}^{j+1}|}{k}\right)\\
&\le\frac{C}{(i\wedge j)^\alpha}\left(|\Delta A|_\alpha+|[\N,\Delta A]|_\alpha\right)||B||, \qquad\text{\underline{by Hölder's inequality and Lemma \ref{explanationMplus}.}}
\end{align*}

Now turn to the last one $\Delta_3$. Similarly, using the same notaion as in Lemma \ref{explanationMplus} we obtain 
\begin{align*}
\Delta_3&\le \sum_{k\ge1}\frac{\overline{A}_i^k}{1+|i-k|}\cdot\frac{|\Delta B|_\alpha}{(k\wedge j)^\alpha}\le |\Delta B|_\alpha\left(\sum_{k\ge j}+\sum_{k\le j}\right)\frac{\overline{A}_i^k}{(1+|i-k|)(k\wedge j)^\alpha}\\
&\le \frac{|\Delta B|_\alpha}{j^\alpha}\sum_{k\ge1}\frac{\overline{A}_i^k}{1+|i-k|}+|\Delta B|_\alpha \left(\sum_{k\ge i/2}+\sum_{k\le i/2}\right)\frac{\overline{A}_i^k}{(1+|i-k|)k^\alpha}\\
&\le \frac{C}{(i\wedge j)^\alpha}|\Delta B|_\alpha\left(||A||+||[\N,A]||\right), \qquad\text{\underline{by Hölder's inequality and Lemma \ref{explanationMplus}.}}
\end{align*}
Here the way to estimate the second series in the second line is the same as the one used for $\Delta_2$ above, so we omit it now.
All the above estimations on $\Delta_1,\Delta_2,\Delta_3$ shows that $|\Delta (AB)|_\alpha\le C||A||_{\wa+}\cdot||B||_{\wa}$, which follows $||AB||_{\widehat{\alpha}}\le C||A||_{\widehat{\alpha}+}\cdot||B||_{\widehat{\alpha}}$ due to the previous estimates on operator norms.

Repeating almost the same procedures, we prove that $BA\in\M_{\widehat{\alpha}}$ with $||BA||_{\widehat{\alpha}}\le C||A||_{\widehat{\alpha}+}\cdot||B||_{\widehat{\alpha}}$. This completes the proof of item $(a)$.

$(b)$. We have $A,B\in\M_{\wa+}$, then we are going to estimate 
\[
||AB||_{\widehat{\alpha}+}=\max\{||AB||,||[\N,AB]||, |\Delta(AB)|_\alpha,|[\N,\Delta(AB)]|_\alpha\}.
\]
First, $||AB||\le ||A||\cdot||B||$. By the identity $[\N,AB]=[\N,A]B+A[\N,B]$ we get 
\[
||[\N,AB]||\le ||[\N,A]||\cdot||B||+||A||\cdot||[\N,B]||.
\]
Then we estimate the last two terms on difference matrices. By \eqref{differenceProdoct} we obtain 
\[
|\Delta (AB)_i^j|\le|A_{i+1}^1|\cdot|B_1^{j+1}|+\sum_{k\ge1}|\Delta A_i^k|\cdot |B_{k+1}^{j+1}|+\sum_{k\ge1}|A_i^k|\cdot|\Delta B_k^j|:=\Delta_{(1)}+\Delta_{(2)}+\Delta_{(3)}.
\]
Clearly $|\Delta_{(1)}|\le \frac{||[\N,A]||\cdot||[\N,B]||}{ij}\le\frac{||[\N,A]||\cdot||[\N,B]||}{(i\wedge j)^\alpha(1+|j-i|)}$. Following the notation in Lemma \ref{explanationMplus} let $\overline{\Delta A}_i^k=|\Delta  A_i^k|(1+|i-k|)$ then we get 
\begin{align*}
\Delta_{(2)}&\le\sum_{k\ge1}\frac{\overline{\Delta A}_i^k}{1+|i-k|}\cdot\frac{\overline{B}_{k+1}^{j+1}}{1+|k-j|}\le\left(|\Delta A|_\alpha+|[\N,\Delta A]|_\alpha\right)\sum_{k\ge1}\frac{\overline{B}_{k+1}^{j+1}}{(i\wedge k)^\alpha(1+|i-k|)(1+|k-j|)}\\
&\le\left(|\Delta A|_\alpha+|[\N,\Delta A]|_\alpha\right)\left(\sum_{1+|i-k|\ge(1+|i-j|)/2}+\sum_{1+|k-j|\ge(1+|i-j|)/2}\right)\frac{\overline{B}_{k+1}^{j+1}}{(i\wedge k)^\alpha(1+|i-k|)(1+|k-j|)}\\
&\le\frac{2(|\Delta A|_\alpha+|[\N,\Delta A]|_\alpha)}{1+|i-j|}\sum_{k\ge1}\left(\frac{\overline{B}_{k+1}^{j+1}}{(i\wedge k)^\alpha(1+|k-j|)}+\frac{\overline{B}_{k+1}^{j+1}}{(i\wedge k)^\alpha(1+|i-k|)}\right).
\end{align*}
Analogous to the $\Delta_2$ and $\Delta_3$ cases, we derive
\[
\Delta_{(2)}\le \frac{C}{(i\wedge j)^\alpha(1+|i-j|)}\left(|\Delta A|_\alpha+|[\N,\Delta A]|_\alpha\right)\cdot\left(||B||+||[\N,B]||\right).
\]
In the same way, letting $\overline{\Delta B}_k^j=|\Delta B_k^j|(1+|k-j|)$, we have 
\begin{align*}
\Delta_{(3)}&\le \sum_{k\ge1}\frac{\overline{A}_i^k}{1+|i-k|}\cdot\frac{\overline{\Delta B}_k^j}{1+|k-j|}\le\left(|\Delta B|_\alpha+|[\N,\Delta B]|_\alpha\right)\sum_{k\ge1}\frac{\overline{A}_i^k}{(k\wedge j)^\alpha(1+|i-k|)(1+|k-j|)}.
\end{align*}
Repeating the estimation procedure used for $\Delta_{(2)}$, we obtain 
\[
\Delta_{(3)}\le\frac{C}{(i\wedge j)^\alpha(1+|i-j|)}\left(|\Delta B|_\alpha+|[\N,\Delta B]|_\alpha\right)\cdot\left(||A||+||[\N,A]||\right).
\]
All the above estimates on $\Delta_{(1)},\Delta_{(2)},\Delta_{(3)}$ imply that 
\[
|\Delta (AB)|_\alpha+|[\N,\Delta (AB)]|_\alpha\le C||A||_{\wa+}\cdot||B||_{\wa+},
\]
which completes the proof of item $(b)$ due to the previous estimates on operator norms.

$(c)$. Since $A\in\M_+$, by definition we have $||A||:=||A||_{\B(\ell^2_0)}\le ||A||_+$, which implies that item $(c)$ holds for $s=0$. Let $s\in(0,2]$. By Lemmas \ref{explanationMplus}, \ref{convergenceSeries}, we get for all $u\in\ell^2_s$
\begin{align*}
||Au||_s^2&=\sum_{i\ge1}i^s|\sum_{j\ge1}A_i^ju_j|^2\le\sum_{i\ge1}\left(\sum_{j\ge1}\frac{\overline{A}_i^j}{1+|i-j|}\left(\dfrac ij\right)^{\frac s2}j^{\frac s2}|u_j|\right)^2\\
&\le \sum_{i\ge1}\left(\sum_{j\ge1}\frac{(i/j)^s}{(1+|i-j|)^2}\right)\left(\sum_{j\ge1}(\overline{A}_i^j)^2j^s|u_j|^2\right)\qquad\text{\underline{by Hölder's inequality}}\\
&\le C\sum_{i\ge1}\sum_{j\ge1}(\overline{A}_i^j)^2j^s|u_j|^2=C\sum_{j\ge1}j^s|u_j|^2\sum_{i\ge1}(\overline{A}_i^j)^2\qquad\text{\underline{by Lemma \ref{convergenceSeries}}}\\
&\le C||A||_+^2\sum_{j\ge1}j^s|u_j|^2=C||A||_+^2||u||_s^2,\qquad\text{\underline{by Lemma \ref{explanationMplus}.}}
\end{align*}
Also, we have for all $u\in\ell^2_{-s}$
\begin{align*}
	||Au||_{-s}^2&=\sum_{i\ge1}i^{-s}|\sum_{j\ge1}A_i^ju_j|^2\le\sum_{i\ge1}\left(\sum_{j\ge1}\frac{\overline{A}_i^j}{1+|i-j|}\left(\dfrac ji\right)^{\frac s2}j^{-\frac s2}|u_j|\right)^2\\
	&\le \sum_{i\ge1}\left(\sum_{j\ge1}(\overline{A}_i^j)^2\right)\left(\sum_{j\ge1}\frac{(j/i)^s}{(1+|i-j|)^2}j^{-s}|u_j|^2\right)\qquad\text{\underline{by Hölder's inequality}}\\
	&\le C||A||_+^2\sum_{i\ge1}\sum_{j\ge1}\frac{(j/i)^s}{(1+|j-i|)^2}j^{-s}|u_j|^2 \qquad\text{\underline{by Lemma \ref{explanationMplus}}}\\
	&=C||A||_+^2\sum_{j\ge1}j^{-s}|u_j|^{2}\sum_{i\ge1}\frac{(j/i)^s}{(1+|j-i|)^2}\\
	&\le C||A||_+^2\sum_{j\ge1}j^{-s}|u_j|^2=C||A||_+^2||u||_{-s}^2,\qquad\text{\underline{by Lemma \ref{convergenceSeries}.}}
\end{align*}
By connecting the two estimates above, we conclude that $||A||_{\B(\ell^2_s)}\le C||A||_+$ for any $s\in[-2,2]$.  Hence, item (c) is proved.

$(d)$. Since $\Delta A\in\M_{\alpha}$, then by definition one gets $|d_{i+1}-d_i|=|\Delta A_i^i|\le\frac{|\Delta A|_{\alpha}}{i^\alpha}$. Without loss of generality, assuming that $i\le j$, we obtain
\[
|d_i-d_j|=\left|\sum_{l=i}^{j-1}d_{l+1}-d_l\right|\le\sum_{l=i}^{j-1}|d_{l+1}-d_l|\le\sum_{l=i}^{j-1}\frac{|\Delta A|_\alpha}{l^\alpha}\le\frac{|\Delta A|_{\alpha}|i-j|}{(i\wedge j)^\alpha}.
\]
The whole proof is now complete.

\end{proof}

\section{Some auxiliary lemmas}

\begin{lemma}\label{explanationMplus}
Given a matrix $A\in\M_+$, define $\overline{A}$ by 
\[
\overline{A}_i^j:=|A_i^j|\cdot(1+|i-j|),\quad i,j\in\bN.
\]
Then we have 
\begin{gather*}
\sqrt{\sum_{i\ge1}|A_i^j|^2},\sqrt{\sum_{j\ge1}|A_i^j|^2}\le ||A||,\\
\sqrt{\sum_{i\ge1}(\overline{A}_i^j)^2},\sqrt{\sum_{j\ge1}(\overline{A}_i^j)^2}\le ||A||+||[\N,A]||.
\end{gather*}
\end{lemma}

\begin{proof}
Let $\{\boldsymbol e_j\}_{j\ge1}$ be an orthonormal basis of $\ell^2$. Then one has $A\boldsymbol{e}_j=(A_i^j)_{i\ge1}$ and therefore 
\[
\sqrt{\sum_{i\ge1}|A_i^j|^2}=||A\boldsymbol{e}_j||\le ||A||.
\]
On the other hand, since $\ell^2$ is  a Hilbert space then we have $||A'||=||A||$, which implies 
\[
\sqrt{\sum_{j\ge1}|A_i^j|^2}=||A'\boldsymbol{e}_i||\le ||A'||=||A||.
\]
By Minkowski's inequality, we obtain
\[
\sqrt{\sum_{i\ge1}(\overline{A}_i^j)^2}=\sqrt{\sum_{i\ge1}(|A_i^j|+|(i-j)A_i^j|)^2}\le \sqrt{\sum_{i\ge1}|A_i^j|^2}+\sqrt{\sum_{i\ge1}|(i-j)A_i^j|^2}\le ||A||+||[\N,A]||.
\]
The remaining estimate is proved similarly.
\end{proof}

\begin{lemma}\label{explanationM}
	Given a matrix $A\in\M$, define $\underline{A}$ by
	\[
	\underline{A}_i^j:=\frac{|A_i^j|}{1+|i-j|}.
	\]
	Then one has $\underline{A}\in\M$ and $||\underline{A}||\le C||A||$.
\end{lemma}

\begin{proof}
	For any $u\in\ell^2$, we have
	\begin{align*}
		||\underline{A}u||^2&=\sum_{i\ge1}|\sum_{j\ge1}\underline{A}_i^ju_j|^2=\sum_{i\ge1}\left(\sum_{j\ge1}\frac{|A_i^j|}{1+|i-j|}|u_j|\right)^2\\
		&\le\sum_{i\ge1}\left(\sum_{j\ge1}\frac{1}{(1+|i-j|)^2}\right)\left(\sum_{j\ge1}|A_i^j|^2|u_j|^2\right)\qquad\underline{\text{by  Hölder's inequality}}\\
		&\le C\sum_{i=1}\sum_{j\ge1}|A_i^j|^2|u_j|^2=C\sum_{j\ge1}|u_j|^2\sum_{i\ge1}|A_i^j|^2\le C||A||^2||u||^2,
	\end{align*}
and thus  $||\underline{A}||\le C||A||$.
\end{proof}

\begin{lemma}\label{convergenceSeries}
Given $s\in[0,2]$, there is a constant $C>0$ such that for all $i,j\in\bN$
\[
\sum_{j\ge1}\frac{(i/j)^s}{(1+|j-i|)^2}\le C,\quad \sum_{i\ge1}\frac{(j/i)^s}{(1+|j-i|)^2}\le C.
\]
\end{lemma}
\begin{proof}
We directly calculate that 
\begin{align*}
&\sum_{j\ge1}\frac{(i/j)^s}{(1+|j-i|)^2}\le\left(\sum_{j\ge i/2}+\sum_{j\le i/2}\right)\frac{(i/j)^s}{(1+|j-i|)^2}\\
\le&\sum_{j\ge i/2}\frac{C}{(1+|j-i|^2)}+\sum_{j\le i/2}\frac{C}{i^{2-s}j^s}
\le C+\sum_{j\le i/2}\frac{C}{j^2}\le C.
\end{align*}
We have proved the first series converges uniformly in $i\in\bN$. Switching the indexes $i,j$ leads to the second estimate.
\end{proof}

\begin{lemma}[see \cite{koch05}]\label{kochLp}
	Let $\lambda^2$ be an eigenvalue and $h$ be the corresponding normalized eigenfunction of the Hermite operator $\bH=-\frac{d^2}{dx}+x^2$. Then we have for any $\delta\in[0,1]$
	\[
	||(1+|x|)^\delta h(x)||_{L^2(\bR)}\lesssim\lambda^{\delta}.
	\]
\end{lemma}
\begin{proof}
	Reasoning as in \cite[Proof of Theorem 3]{koch05} we set $D_\lambda=\{x\in\bR:|x|\le 2\lambda\}$ and denote its complement by $D_\lambda^c$. By Minkowski's inequality, we have 
	\[
	||(1+|x|)^\delta h(x)||_{L^2(\bR)}\le ||(1+|x|)^\delta h(x)||_{L^2(D_\lambda)}+||(1+|x|)^\delta h(x)||_{L^2(D_\lambda^c)}:=L_1+L_2.
	\] 
	Clearly, by Hölder's inequality we get $L_1\le||(1+|x|)^\delta||_{L^\infty(D_\lambda)}||h(x)||_{L^2(D_\lambda)}\lesssim \lambda^\delta$. Then by estimate (11) in \cite{koch05} one obtains \[
	||(x^2-\lambda^2)h(x)||_{L^2(D_\lambda^c)}\lesssim \lambda^{1/3},
	\]
and thus $L_2\lesssim\lambda^{-2/3}\lesssim \lambda^\delta$. This completes the proof.
\end{proof}

Following almost the same idea, we directly get the following result.

\begin{lemma}\label{logdelta}
	Let $\lambda^2$ be an eigenvalue and $h$ be the corresponding normalized eigenfunction of the Hermite operator $\bH=-\frac{d^2}{dx}+x^2$. Then we have for any $\delta'\ge0$
	\[
	||(1+|x|)\ln^{-\delta'}(2+|x|)h(x)||_{L^2(\bR)}\lesssim\frac{\lambda}{\ln^{\delta'}(2+\lambda)}.
	\]
\end{lemma}

\begin{lemma}\label{measure}
Let $f:[0,1]\mapsto\bR$ be a $C^1$ function whose derivative satisfies $|f'(x)|\ge\varsigma>0$ for all $x\in[0,1]$. Then, for each $\kappa>0$, one has $\meas\big(\{x\in[0,1]:|f(x)|<\kappa\}\big)\le\frac{2\kappa}{\varsigma}$.
\end{lemma}

\begin{lemma}\label{Melnikov}
Let $\lambda_i=2i-1$ for $i\in\bN$. Then there are positive constants $C,\mu_1,\mu_2$ such that for all $\gamma\in(0,1/4)$ and $K\ge1$ there exists a closed subset $\mathbb D=\mathbb D(\gamma,K)\subset\Pi$, satisfying $\meas(\Pi\backslash\mathbb D)\le C\gamma^{\mu_1}K^{\mu_2}$, such that the following holds for all $\omega\in\mathbb D$:
\[
|k\cdot\omega+\lambda_i-\lambda_j|\ge\gamma(1+|i-j|),\quad\forall\,i,j\in\bN \text{ and }\forall\,k\in\bZ^n\setminus\{0\},~|k|\le K.
\]
\end{lemma}
\begin{proof}
Define for $l\in\bZ$ and $k\in\bZ^n\setminus\{0\}$ the set 
\[
E_l^k(\gamma):=\{\omega\in\Pi:|k\cdot\omega+l|\ge\gamma(1+|l|)\}.
\]
By Lemma \ref{measure} above, there is a constant $C=C(n)>0$ such that 
\[
\meas\left(\Pi\backslash E_l^k(\gamma)\right)\le C(n)\frac{\gamma(1+|l|)}{|k|}.
\]
Notice that for all $k\in\bZ^n\setminus\{0\}$ and $|l|\ge 4\pi|k|$ one has $E^k_l(\gamma)=\Pi$ since $\gamma\in(0,1/4)$.  Given that $\lambda_i-\lambda_j\in\bZ$, we define $\mathbb D:=\bigcup_{\substack{l\in\bZ\\ k\in\bZ^n,0<|k|\le K}}E^k_l(\gamma)$. Lemma \ref{measure} shows that 
\[
\meas(\Pi\backslash\mathbb D)\le\sum_{\substack{l\in\bZ,|l|\le 4\pi|k|\\ k\in\bZ^n,0<|k|\le K}}\meas\left(\Pi\backslash E^k_l(\gamma)\right)\le\sum_{k\in\bZ^n,0<|k|\le K}C(n)\gamma|k|\le C(n)\gamma K^{n+1}.
\]
Setting $\mu_1=1$ and $\mu_2=n+1$, we have the thesis.
\end{proof}

\subsubsection*{\bf Acknowledgments.}
Both authors were supported by PRIN 2022 ``Turbulent effects vs Stability in Equations from Oceanography'' (acronym: TESEO), project number: 2022HSSYPN.
Besides, Z. Wang was partially supported by China Postdoctoral Science Foundation (Grant no.\,2024M751512) and Nankai Zhide Foundation. 
Furthermore, E. Haus would like to thank INdAM-GNAMPA, and
Z. Wang would like to thank Zhenguo Liang for helpful discussions on technical details, as well as for suggestions regarding background and relevant literature.

\subsubsection*{\bf Declaration of competing interests} 
The authors have no competing interests to declare related to this work.


\end{document}